\documentclass[11pt,a4paper]{amsart}

\usepackage{latexsym}
\usepackage{amsmath}
\usepackage{amsthm}
\usepackage{amssymb}
\usepackage{vmargin}
\usepackage{amscd}
\usepackage{stmaryrd}
\usepackage{euscript}
\usepackage{mathrsfs}
\usepackage{amscd}
\usepackage[all]{xy}
\usepackage{xr}

\DeclareMathAlphabet{\mathpzc}{OT1}{pzc}{m}{it}

\newtheorem{theorem}{Theorem}[section]
\newtheorem{proposition}[theorem]{Proposition}
\newtheorem{corollary}[theorem]{Corollary}

\newtheorem{lemma}[theorem]{Lemma}
\newtheorem*{theorem*}{Theorem}
\newtheorem*{proposition*}{Proposition}
\newtheorem*{corollary*}{Corollary}
\newtheorem*{lemma*}{Lemma}
\newtheorem*{conjecture*}{Conjecture}

\theoremstyle{definition}
\newtheorem{definition}[theorem]{Definition}

\newtheorem*{definition*}{Definition}

\theoremstyle{remark}
\newtheorem{example}[theorem]{Example}

\newtheorem{remark}[theorem]{Remark}

\newtheorem*{example*}{Example}
\newtheorem*{examples*}{Examples}
\newtheorem*{remark*}{Remark}
\newtheorem*{remarks*}{Remarks}
\newtheorem*{exercise*}{Exercise}

\newcommand\da{\!\downarrow\!}

\newcommand\la{\leftarrow}

\newcommand\lla{\longleftarrow}
\newcommand\id{\mathrm{id}}

\newcommand\ten{\otimes}

\newcommand\DD{\mathrm{D}}

\renewcommand\H{\mathrm{H}}
\newcommand\z{\mathrm{Z}}
\renewcommand\b{\mathrm{B}}

\newcommand\N{\mathbb{N}}
\newcommand\Z{\mathbb{Z}}
\newcommand\Q{\mathbb{Q}}

\newcommand\bI{\mathbb{I}}

\newcommand\bL{\mathbb{L}}

\newcommand\bS{\mathbb{S}}

\newcommand\C{\mathcal{C}}

\newcommand\cD{\mathcal{D}}

\newcommand\cN{\mathcal{N}}

\renewcommand\O{\mathscr{O}}

\newcommand\sP{\mathscr{P}}

\newcommand\fX{\mathfrak{X}}

\renewcommand\L{\Lambda}

\newcommand\m{\mathfrak{m}}

\newcommand\fp{\mathfrak{p}}

\newcommand\ext{\mathscr{E}\!\mathit{xt}}

\newcommand\Ho{\mathrm{Ho}}
\newcommand\Ring{\mathrm{Ring}}
\newcommand\Alg{\mathrm{Alg}}

\newcommand\Mod{\mathrm{Mod}}

\newcommand\Hom{\mathrm{Hom}}

\newcommand\HHom{\underline{\mathrm{Hom}}}

\newcommand\Ext{\mathrm{Ext}}

\newcommand\cone{\mathrm{cone}}

\newcommand\coker{\mathrm{coker\,}}

\newcommand\im{\mathrm{Im\,}}

\newcommand\Ab{\mathrm{Ab}}

\newcommand\Spec{\mathrm{Spec}\,}
\newcommand\Dec{\mathrm{Dec}\,}
\newcommand\DEC{\mathrm{DEC}\,}

\newcommand\Set{\mathrm{Set}}

\newcommand\Lim{\varprojlim}
\newcommand\LLim{\varinjlim}
\DeclareMathOperator*{\holim}{holim}

\newcommand\into{\hookrightarrow}
\newcommand\onto{\twoheadrightarrow}

\newcommand\xra{\xrightarrow}

\newcommand\bt{\bullet}
\newcommand\by{\times}

\newcommand\Tot{\mathrm{Tot}\,}
\newcommand\diag{\mathrm{diag}\,}

\newcommand\ev{\mathrm{ev}}
\newcommand\ind{\mathrm{ind}}
\newcommand\pro{\mathrm{pro}}

\newcommand\pd{\partial}

\newcommand\gpd{\mathrm{Gpd}}

\newcommand\sk{\mathrm{sk}}

\newcommand\op{\mathrm{opp}}

\newcommand\oR{\mathbf{R}}

\newcommand\uleft\underleftarrow
\newcommand\uline\underline
\newcommand\uright\underrightarrow

\begin{document}

\begin{abstract}
Lurie's representability theorem gives necessary and sufficient conditions for a functor to be an almost  finitely presented derived geometric stack.
 We establish several variants of Lurie's theorem, making the hypotheses easier to verify for many applications.
Provided a derived analogue of Schlessinger's condition holds, the theorem reduces to verifying conditions on the underived part and on cohomology groups.  Another simplification is that functors need only be defined on nilpotent extensions of discrete rings. Finally, there is a  pre-representability theorem, which can be applied to associate explicit geometric stacks to dg-manifolds and related objects.
\end{abstract}

\title{Representability of derived stacks}

\author{J.P.Pridham}
\thanks{The author was supported during this research by  the Engineering and Physical Sciences Research Council [grant number  EP/F043570/1].}
\maketitle

\tableofcontents

\section*{Introduction}

Artin's representability theorem (\cite{Artin}) gives necessary and sufficient conditions for a functor from $R$-algebras to groupoids to be representable by an algebraic Artin stack, locally of finite presentation. In his thesis, Lurie established a similar result not just for derived Artin $1$-stacks, but for derived geometric Artin $n$-stacks.  Explicitly, given a functor $F:s\Alg_R \to \bS$ from simplicial $R$-algebras to simplicial sets, \cite{lurie}  Theorems 7.1.6 and 7.5.1 give necessary and sufficient conditions for $F$ to be representable by a derived geometric Artin $n$-stack, almost of finite presentation over $R$. 

Lurie's Representability Theorem is more natural than Artin's in one important respect: in the derived setting, existence of a functorial obstruction theory is an automatic consequence of left-exactness. 
However, Lurie's theorem can be difficult to verify for problems not explicitly coming from topology. The most basic difficulty can be showing that a functor is homotopy-preserving, or finding a suitable functor which is. It tends to be even more difficult to show that a functor is almost of finite presentation, or to verify that it is a hypersheaf. 
The purpose of this paper is to adapt the representability theorems in \cite{lurie} and \cite{hag2}, simplifying these criteria for a functor $F:s\Alg_R \to \bS$ to be a geometric $n$-stack. 

In \cite{lurie}, the key exactness properties used were  cohesiveness and infinitesimal cohesiveness. These are said to hold for a functor $F: s\Alg_R \to \bS$ if the maps
\[
\theta: F(A\by_BC) \to F(A)\by^h_{F(B)}F(C)
\]
to the homotopy fibre product are weak equivalences for all surjections (resp. nilpotent surjections) $A \onto B$ and $C \onto B$. The key idea of this paper is to introduce a notion more in line with Schlessinger's conditions (\cite{Sch}). We say that $F$ is homotopy-homogeneous if $\theta$ is a weak equivalence for all nilpotent surjections $A \onto B$ and arbitrary maps $C \to B$. 

The first major consequence is Theorem \ref{cohofp}, showing that if $F$ is homotopy-homogeneous, then it is almost finitely presented whenever the restriction $\pi^0(F):\Alg_{\H_0R} \to \bS$ and the cohomology theories $\DD^i_x(F,-)$ of the tangent spaces of $F$ at discrete points $x$ are all finitely presented. This reduces the question to familiar invariants, since the  cohomology groups are usually naturally  associated to the moduli problem. Likewise, Proposition \ref{sheafresult} shows that to ensure that a homotopy-homogeneous functor $F$ is a hypersheaf, it suffices to check that $\pi^0F$ is a hypersheaf and that the 
modules $\DD^i_x(F,-)$ are quasi-coherent.

These results are applied to 
Proposition \ref{cotexists}, which  shows that with certain additional finiteness hypotheses on $\DD^i_x(F)$, a cotangent 
complex and obstruction theory exist for $F$. This leads to Theorem \ref{lurierep}, which replaces Lurie's almost finite presentation condition with those of Theorem \ref{cohofp}. We then obtain Corollary \ref{lurierep2}, which incorporates the further simplifications of  Proposition \ref{sheafresult}.

A key principle in derived algebraic geometry is that the derived structure is no more than an infinitesimal thickening of the underived objects. For instance, every simplicial ring can be expressed as a composite of homotopy square-zero extensions of a discrete ring. Proposition \ref{cNhat} strictifies this result, showing that we can  work with extensions which are nilpotent (rather than just homotopy nilpotent). This approach leads to Theorem \ref{lurierep3}, which shows how the earlier representability results can be reformulated for functors on dg or simplicial rings $A$ for which $A \to \H_0A$ is nilpotent, thereby removing the need for Lurie's nilcompleteness hypothesis.

The last major result is Theorem \ref{lurieprerep}, which shows how to construct representable functors from functors which are not even homotopy-preserving. The key motivation is Example \ref{dgexample}, which constructs explicit derived geometric stacks from Kontsevich's dg manifolds.

The structure of the paper is as follows.

In Section \ref{rep}, we recall Lurie's Representability Theorem, introduce homotopy-homogeneity, and establish the variants Theorem \ref{lurierep} and Corollary \ref{lurierep2} of Lurie's theorem. We also establish Proposition \ref{detectweak}, which identifies weak equivalences between geometric derived $n$-stacks, and Proposition \ref{sqc}, which gives a functorial criterion for strong quasi-compactness.

Section \ref{completesn} then introduces simplicial or dg algebras $A$ for which $A \to \H_0A$ is a nilpotent extension, showing in Theorem \ref{lurierep3} how to re-interpret representability in terms of functor on such algebras.

Finally, Section \ref{prerepsn} introduces the notion of homotopy-surjecting functors; these map square-zero acyclic extensions to  surjections. For any such functor $F$, we construct another functor $\bar{W}\uline{F}$, and  Proposition \ref{settotop} shows that this is homotopy-preserving whenever $F$ is homotopy-homogeneous and homotopy-surjecting. This leads to Theorem \ref{lurieprerep}, which gives sufficient conditions on $F$ for $\bar{W}\uline{F}$ to be a derived geometric $n$-stack.

\section{Representability of derived stacks}\label{rep}

We denote the category of simplicial sets by $\bS$, the category of simplicial rings by $s\Ring$, and the category of simplicial $R$-algebras by $s\Alg_R$. We let $dg_+\Alg_R$ be the category of differential graded-commutative $R$-algebras in non-negative chain degrees. The homotopy category $\Ho(\C)$ of a category $\C$ is obtained by formally inverting weak equivalences.

\subsection{Background}

Given a simplicial ring $R$, a derived geometric $n$-stack over $R$ is a functor
\[
 F:s\Alg_R \to \bS
\]
satisfying many additional conditions.  These are detailed in  \cite{hag2} Chapter 2.2 or \cite{lurie} \S 5.1. A more explicit characterisation in terms of certain simplicial cosimplicial rings is given in \cite{stacks2} Theorem \ref{stacks-bigthm}. However, for the purposes of this paper, these definitions are largely superfluous, since  
it will be enough to consider  functors satisfying Lurie's Representability  Theorem: 

\begin{theorem}\label{lurierep0}
  A homotopy-preserving functor $F: s\Alg_R \to \bS$ is a geometric derived $n$-stack which is almost of finite presentation if and only if 
\begin{enumerate}
 \item The functor $F$ commutes with filtered colimits when restricted to $k$-truncated objects of $s\Alg_R$, for
each $k \ge  0$.
\item For any discrete commutative ring $A$, the space $F(A)$ is $n$-truncated.
\item The functor $F$ is a hypersheaf for the \'etale topology.
\item The functor $F$ is cohesive:
for any pair $A \to C$, $B \to C$ of surjective morphisms in $s\Alg_R$, the
induced map 
\[
 F(A \by_C B) \to F(A) \by_{F(C)}^h F(B)
\]
 is a weak equivalence.
\item The functor $F$ is nilcomplete:
for any $A \in  s\Alg_R$, the natural map $F(A) \to \Lim^h_kF(P_kA)$ is an equivalence, where $\{P_kA\}_k$ denotes the Moore-Postnikov tower of $A$.

\item Let $B$ be a complete, discrete, local, Noetherian $R$-algebra, and $\m \subset B$ the maximal ideal. Then the
natural map $F(B) \to \Lim_n^h F(B/m^n)$ is a weak equivalence.

\item  Let $x \in  F(C)$, where $C$ is a (discrete) integral domain which is finitely
generated as a $\pi_0R$-algebra. 
 For each $i,n$, the tangent module 
\[
 \DD^{n-i}_x(F,C):= \pi_i (F(C\oplus C[-n])\by^h_{F(C)}\{x\})
\]
 is a finitely generated $C$-module.

\item $R$ is a derived G-ring:
\begin{enumerate}
  \item $\pi_0R$ is Noetherian, 
  \item for each prime ideal $\fp \subset \pi_0R$, the $\fp(\pi_0R)_{\fp}$-adic completion of $(\pi_0R)_{\fp}$ is a geometrically regular $\pi_0R$-algebra, and
  \item for all $n$, $\pi_nR$ is a finite $\pi_0R$-module.
\end{enumerate}

\item $R$ admits a dualising module in the sense of
 \cite{lurie} Definition 3.6.1. [For discrete rings, this is equivalent to  a dualising complex. In particular,  $\Z$ and   Gorenstein local rings  are all derived G-rings with dualising modules.]
\end{enumerate}
\end{theorem}
\begin{proof}
\cite{lurie} Theorem 7.5.1.
\end{proof}
 Readers unfamiliar with the conditions  of this theorem should not despair, since the conditions will be explained and considerably simplified over the course of this paper.

\begin{remark}\label{cflurie}
Note that there are slight differences in terminology between \cite{hag2} and \cite{lurie}. In the former, only disjoint unions of affine schemes are $0$-representable, so arbitrary schemes are $2$-geometric stacks, and Artin stacks are $1$-geometric stacks if and only if they have affine diagonal. In the latter, algebraic spaces are $0$-stacks.  A geometric $n$-stack  is called $n$-truncated in \cite{hag2}, and it follows easily that every $n$-geometric stack in \cite{hag2} is $n$-truncated. Conversely, every  geometric $n$-stack  is  $(n+2)$-geometric. 

We can summarise this by saying that for a derived geometric stack $\fX$ to be $n$-truncated means that $\fX \to \fX^{S^{n+1}}$ is an equivalence, or equivalently that $\fX \to \fX^{S^{n-1}}$ is representable by derived algebraic spaces. For $\fX$ to be $n$-geometric means that $\fX \to \fX^{S^{n-1}}$ is representable by disjoint unions of derived affine schemes.

Theorem \ref{lurierep} takes the convention from \cite{lurie}, so ``geometric derived $n$-stack'' means ``$n$-truncated derived geometric stack''.
\end{remark}

\subsection{Tangent spaces and homogeneity}\label{tgtsn}

\begin{definition}\label{sq0def}
We say that  a map $A \to B$ in $s\Ring$ is a square-zero extension if it is surjective, and the kernel $I$ is square-zero, i.e. satisfies $I^2=0$. 
\end{definition}

\begin{lemma}\label{luriesmall}
In $\Ho(s\Alg_R)$, square-zero extensions $A \to  B$ with kernel $I$ correspond up to weak equivalence to the small extensions $A$ of $B$ by $I$ in the sense of \cite{lurie} Definition 3.3.1.
\end{lemma}
\begin{proof}
Given a square-zero extension $A \to B$, observe that the kernel $I$ is a simplicial $B$-module. Choose an inclusion $i:I \into N$ of simplicial $B$-modules, with $N$ acyclic, and set $\tilde{B}$ to be the simplicial algebra $A\oplus_{I}N$. Then $\tilde{B} \to B$ is a trivial fibration, and if we let $C=\coker i$, then
$$
A= \tilde{B}\by_{B\oplus C} B.
$$

Now we need only observe that $\Omega C \simeq I$ in the notation of \cite{hag2}, so $\tilde{B} \to   B\oplus C$ gives a homotopy derivation $s: B \to I[-1]$, with 
$$
A= B\oplus_{s}I:= B\by^h_{\id +s, B\oplus I[-1],\id +0 }B,
$$
so $A \to B$ is a small extension in Lurie's sense.

Conversely, given a homotopy derivation $s: B \to M[-1]$, we may assume that $B$ is cofibrant, so lift this to a morphism $B \to B \oplus  M[-1]$ of simplicial $R$-algebras. Taking a surjection $f:N \onto  M[-1]$ of simplicial $B$-modules, with $N$ acyclic, we see that
$$
B\by^h_{\id +s, B\oplus M[-1],\id +0 }B\simeq B\by_{\id +s, B\oplus M[-1],\id}(B\oplus N),
$$
since the right-hand map is a fibration. But this maps surjectively to $B$, with kernel $I:=\ker f$, which is a $B$-module, so square-zero. Moreover $M \simeq I$, so the respective square-zero extensions are by the same module.
\end{proof}

\begin{remark}\label{obskey}
Given a simplicial ring $B$ and a simplicial $B$-module $M$, we may define a derivation $t:B \oplus M[-1] \to M[-1]$ given by $0$ on $B$, and by the identity on $M[-1]$. The corresponding square-zero extension $(B\oplus M[-1])\oplus_t M$ is equivalent to $B$. In particular, this means that $B \to B \oplus M[-1]$ is weakly equivalent to a square-zero extension.
\end{remark}

\begin{definition}
Say that a functor between model categories is homotopy-preserving if it maps weak equivalences to weak equivalences. 
\end{definition}

\begin{definition}\label{hhgsdef}
We say that a 
functor 
$$
F:s\Alg_R \to \bS
$$
is homotopy-homogeneous if  for all square-zero extensions $A \to B$ and all maps $C \to B$ in $s\Alg_R$, the natural map
$$
F(A\by_BC) \to F(A)\by^h_{F(B)}F(C)
$$
to the homotopy fibre product is a weak  equivalence.
\end{definition}

\begin{definition}\label{Tdef}
Given a homotopy-preserving homotopy-homogeneous  functor $F:s\Alg_R \to \bS$, a simplicial $R$-algebra $A$, and a point $x \in F(A)$, define the tangent functor
$$
T_x(F/R): s\Mod_A \to \bS
$$
by
$$
T_x(F/R)(M):= F(A\oplus M)\by^h_{F(A)}\{x\}.
$$
\end{definition}

\begin{lemma}\label{pi0tan}
If $F$ satisfies the conditions of Definition \ref{Tdef}, then up to non-canonical weak equivalence, $T_x(F/R)(M)$ is an invariant of the class $[x] \in \pi_0F(A)$. 
\end{lemma}
\begin{proof}
Given a path $\gamma:\Delta^1 \to F(A)$, we have equivalences
$$
T_{\gamma(0)}(F/R)(M)\simeq \Delta^1\by_{\gamma, F(A)}^hF(A\oplus M) \simeq T_{\gamma(1)}(F/R)(M),
$$
so paths in $F(A)$ give equivalences between stalks. Considering maps $\Delta^2 \to F(A)$, we see that these equivalences satisfy the cocycle condition up to homotopy, with the maps $\Delta^n \to F(A)$ giving higher homotopies.
Thus $T_{(-)}(F/R)(M)$ forms a weak local coefficient system on $F(A)$.
\end{proof}

\begin{definition}\label{normdef}
 Given a simplicial abelian group $A_{\bt}$, we denote the associated normalised chain complex  by $NA$. Recall that this is given by  $N(A)_n:=\bigcap_{i>0}\ker (\pd_i: A_n \to A_{n-1})$, with differential $\pd_0$. Then $\H_*(NA)\cong \pi_*(A)$.

Using the Eilenberg-Zilber shuffle product,  normalisation  $N$ extends to a functor
$$
N:s\Ring \to dg_+\Ring
$$ 
from simplicial rings to differential graded rings in non-negative chain degrees.
\end{definition}

By the Dold-Kan correspondence, normalisation gives an equivalence of categories between simplicial abelian groups and chain complexes in non-negative  degrees. For any $R \in s\Ring$, this extends to an equivalence
$$
s\Mod_R \to dg_+\Mod_{NR}
$$
between simplicial $R$-modules and $NR$-modules in non-negatively graded chain complexes.

\begin{definition}\label{suspdef} 
Given a chain complex $V$, let $V[r]$ be the chain complex $V[r]_i:= V_{r+i}$. Given a simplicial abelian group $M$ and $n\ge 0$, let $M[-n]:= N^{-1}(NM[-n])$, where $N^{-1}$ is inverse to the normalisation functor $N$.

For $R \in s\Ring$, observe that this extends to a functor $[-n]: s\Mod_R \to s\Mod_R$. Note that $\pi_iM[-n]= \pi_{i-n}M$.\end{definition}

\begin{lemma}\label{adf}
For all  $F,A,M,x$ as in Definition \ref{Tdef}, there is a natural abelian structure on $\pi_iT_xF(M)$. Moreover,  there are natural isomorphisms
$$
\pi_iT_x(F/R)(M)\cong \pi_{i+1}T_xF( M[-1]),
$$
where homotopy groups are defined relative to the basepoint $0$ given by the image of $T_x(F/R)(0) \to T_x(F/R)(M)$.
\end{lemma}
\begin{proof}
Addition in $M$ gives a  morphism 
$$
(A\oplus M) \by_A(A\oplus M) \cong A\oplus (M\oplus M) \to A \oplus M,
$$ 
 so the corresponding map 
$$
F(A\oplus M)\by^h_{F(A)} F(A\oplus M) \to F(A\oplus M).
$$
induces  an  abelian structure on $\pi_iT_xF(M)$.

For the second part, observe that $M= 0\by^h_{ M[-1]}0$, and that $0 \to M[-1]$ is surjective (in the sense that it is surjective on $\pi_0$), so 
$$
F(A\oplus M) \simeq F(A)\by^h_{F(A\oplus M[-1])}F(A)
$$
by homotopy-homogeneity, giving 
$$
T_x(F/R)(M)\simeq 0\by^h_{T_x(F/R)( M[-1])}0.
$$
Thus $\pi_iT_x(F/R)(M)\cong \pi_{i+1}T_x(F/R)( M[-1])$, as required.
\end{proof}

\begin{definition}\label{totcohodef} 
For all  $F,A,x$ as above, and all simplicial $A$-modules $M$, define
$$
\DD^{n-i}_x(F,M):= \pi_i (T_x(F/R) (M[-n])),
$$
observing that this is well-defined, by Lemma \ref{adf}.
\end{definition}

\begin{remark}
Observe that if $F$ is a derived geometric $n$-stack, and $x: \Spec A \to F$ over $\Spec R$, then $\DD^j_x(F,M)= \Ext^j_A(x^*\bL_{\bt}^{F/\Spec R}, M)$, for $\bL_{\bt}^{F/R}$ the cotangent complex of $F$ over $R$. 
\end{remark}

\begin{lemma}\label{tantrans}
For $F, A,x$ as above, with  $f:A \to B$ a morphism of simplicial $R$-algebras, and $M$ a simplicial $B$-module, there are natural isomorphisms
$$
T_x(F/R)(f_{*}M) \simeq T_{f_*x}(F/R)(M),
$$
and hence $\DD^j_x(F, f_*M) \cong \DD^j_{f_*x}(F, M)$.
\end{lemma}
\begin{proof}
This is just the observation that $A\oplus f_*M = A\by_B (B\oplus M)$, so $F(A\oplus f_*M)\simeq F(A)\by^h_{F(B)}F(B\oplus M)$.
\end{proof}

\begin{lemma} \label{tanlocsys}
If  $ X :s\Alg_R \to \bS$ is homotopy-preserving and homotopy-homogeneous,  take 
an object $A \in s\Alg_R$ and an $A$-module $M$.
Then there is a local coefficient system 
$$
\DD^*(X,M)
$$
on the simplicial set $X(A)$, whose stalk at $x \in X(A)$ is $\DD^*_x(X,M)$. In particular, $\DD^*_x(X,M)$ depends (up to non-canonical isomorphism) only on the image of $x$ in $\pi_0X(A)$.
\end{lemma}
\begin{proof}
This follows straightforwardly from the proof of  Lemma \ref{pi0tan}. 
\end{proof}

\begin{proposition}\label{obs}
If $F:s\Alg_R \to \bS$ is homotopy-preserving and homotopy-homogeneous, then   
  for any square-zero extension $e:I \to A \xra{f} B$ in $ s\Alg_R$, there is a sequence of sets
$$
\pi_0(FA)\xra{f_*} \pi_0(FB) \xra{o_e} \Gamma(FB,\DD^1(F, I)), 
$$  
where $\Gamma(-)$ denotes the global section functor. 
This is exact in the sense that the fibre of $o_e$ over $0$ is the image of $f_*$. 
 Moreover,  there is a group action of $\DD^0_x(F, I)$ on the fibre of $\pi_0(FA) \to \pi_0(FB)$ over $x$, whose orbits are precisely the fibres of $f_*$. 

For any $y \in F_0A$, with $x=f_*y$, the fibre of $FA \to FB$ over $x$ is weakly equivalent to $T_{x}(F/R,I)$, and the sequence above 
extends to a long exact sequence
$$\xymatrix@R=0ex{
\cdots \ar[r]^-{e_*} &\pi_n(FA,y) \ar[r]^-{f_*}&\pi_n(FB,x) \ar[r]^-{o_e}& \DD^{1-n}_{y}(F,I) \ar[r]^-{e_*} &\pi_{n-1}(FA,y)\ar[r]^-{f_*}&\cdots\\ &\cdots \ar[r]^-{f_*}&\pi_1(FB,x) \ar[r]^-{o_e}& \DD^0_{y}(F,I)  \ar[r]^-{-*y} &\pi_0(FA).
}
$$
\end{proposition}
\begin{proof}
 The proof of \cite{ddt1} Theorem \ref{ddt1-robs} carries over to this context. The main idea  is that 
as in the proof of  Lemma \ref{luriesmall}, there is a trivial fibration $\tilde{B} \to B$, and $A= \tilde{B}\by_{B\oplus I[-1]}B$, with $\tilde{B} \to B\oplus I[-1]$ a square-zero extension. By homotopy-homogeneity,
$$
F(A) \simeq F(\tilde{B})\by_{F(B\oplus I[-1])}^h F(B),
$$
and $F(\tilde{B}) \simeq F(B)$ since $F$ is homotopy-preserving.

The rest of the proof then follows by studying the long exact sequence of homotopy groups associated to the homotopy fibres of
$$
F(\tilde{B}) \to F(B\oplus I[-1])
$$
and of $FA\to FB$, noting that $F(A\by_BA) \simeq F(A)\by_{F(B)}^hF(B \oplus I)$.
\end{proof}


\subsection{Finite presentation}

\begin{definition}
Recall (e.g. from \cite{sht} Definition VI.3.4) that the Moore-Postnikov tower $\{P_nX\}$ of a fibrant simplicial set $X$ is given by
$$
P_nX_q:=\im( X_q \to \Hom(\sk_n\Delta^q, X)),
$$  
with the obvious simplicial structure. Here, $\sk_nK$ denotes the $n$-skeleton of $K$, the simplicial set generated by $K_{\le n}$. 

The spaces $P_nX$ form an inverse system $X\to \ldots \to P_nX \to P_{n-1}X \to \ldots$, with $X=\Lim P_nX$, and 
$$
\pi_q P_nX= \left\{ \begin{matrix} \pi_q X & q \le n \\ 0 & q >n. \end{matrix} \right.
$$ 
The maps $P_nX\to P_{n-1}X$ are fibrations. If $X$ is reduced, then so is $P_nX$.
\end{definition}

\begin{definition}
Define $\tau_{\le k} (s\Alg_R)$ to be the full subcategory of $s\Alg_R$ consisting of objects $A$ with 
$A=P_kA$, the $k$th Moore-Postnikov space.
\end{definition}

\begin{definition}
Define 
the category $ \tau_{\le k}\Ho(s\Alg_R)$ to be the full subcategory of $\Ho(s\Alg_R)$ consisting of objects $A$ with $\pi_iA=0$ for $i>k$.  Note that $ \tau_{\le k}\Ho(s\Alg_R)$ is equivalent to the category $\Ho(\tau_{\le k} (s\Alg_R))$ obtained by localising $\tau_{\le k} (s\Alg_R)$ at weak equivalences.
\end{definition}

\begin{definition}
Recall from \cite{lurie} Proposition 5.3.10 that a homotopy-preserving functor $F: s\Alg_R \to \bS$ is said to be almost of finite presentation if for all $k$ and all filtered direct systems $\{A_{\alpha}\}_{\alpha \in \bI}$ in $ \tau_{\le k} (s\Alg_R)$, the map
$$
{\varinjlim} F(A_{\alpha})  \to F(\varinjlim A_{\alpha})
$$
is a weak equivalence.
\end{definition}

\begin{definition}
Given a functor $F: s\Alg_R \to \bS$, define $\pi^0F: \Alg_{\pi_0R} \to \bS$ by $\pi^0F(A)=F(A)$.
\end{definition}

\begin{theorem}\label{cohofp}
If a homotopy-preserving functor $F: s\Alg_R \to \bS$ is homotopy-homogeneous, then it is almost of finite presentation if and only if the following hold:
\begin{enumerate}
\item the functor $\pi^0F: \Alg_{\pi_0R} \to \bS$  preserves filtered colimits;
 
\item for all finitely generated $A \in \Alg_{\pi_0R}$  and all $x \in F(A)_0$, the functors $\DD^i_x(F, -): \Mod_A \to \Ab$ preserve filtered colimits for all $i>0$.
\end{enumerate}
\end{theorem}
\begin{proof}
Note that since $\pi^0F$ preserves filtered colimits, Lemma \ref{adf} implies that  the functors $\DD^i_x(F, -): \Mod_A \to \Ab$  preserve filtered colimits for all $i\le 0$.

We need to prove that $F$ preserves filtered homotopy colimits in the categories $\tau_{\le k} (s\Alg_R)$. We prove this by induction on $k$, the case $k=0$ following by hypothesis. 

Take a filtered direct system  $\{A_{\alpha}\}$ in $\tau_{\le k} (s\Alg_R)$, with homotopy colimit $A$. Let $B_{\alpha}=P_{k-1}A_{\alpha}, B=P_{k-1}A$. Let $M_{\alpha}:= \pi_kA_{\alpha}, M:= \pi_kA$, and observe that these are $\pi_0A_{\alpha}$- and $\pi_0A$-modules respectively.

Now, $A_{\alpha} \to B_{\alpha}$ and $A \to B$ are square-zero extensions up to homotopy (see for instance \cite{hag2} Lemma 2.2.1.1), coming  from essentially unique homotopy derivations $\delta: B_{\alpha} \to M_{\alpha}[-k-1]$, with
$$
A_{\alpha} \simeq B_{\alpha}\by^h_{\id +\delta, B_{\alpha}\oplus M_{\alpha}[-k-1],\id +0 }B_{\alpha}= B_{\alpha}\by^h_{\id +\delta, \pi_0(A_{\alpha})\oplus M_{\alpha}[-k-1],\id +0 } \pi_0(A_{\alpha}).
$$  

Now, by Remark \ref{obskey}, the map $\pi_0(A_{\alpha}) \to  \pi_0(A_{\alpha})\oplus M_{\alpha}[-k-1]$ is weakly equivalent to a square-zero extension.  Thus, since $F$ is homotopy-homogeneous,
$$
F(A_{\alpha}) \simeq F(B_{\alpha})\by^h_{\id +\delta,F(\pi_0(A_{\alpha})\oplus M_{\alpha}[-k-1])}F(\pi_0(A_{\alpha}))
$$
and similarly for $A$. 

We wish to show that $\theta:\varinjlim F(A_{\alpha}) \to F(A)$ is a weak equivalence, and our inductive hypothesis gives $\varinjlim F(B_{\alpha}) \simeq F(B)$. It therefore suffices to consider the homotopy fibre of $\theta$ over $y \in F(B)$, which lifts to some $\tilde{y} \in F(B_{\beta})$. If we let $\tilde{y}_{\alpha}$ be the image of $\tilde{y}$ in $F(B_{\alpha})$, this  gives
$$
\theta_y: \varinjlim \{\tilde{y}_{\alpha}\}\by^h_{\id +\delta,F(\pi_0(A_{\alpha})\oplus M_{\alpha}[-k-1])}F(\pi_0(A_{\alpha})) \to \{y\}\by^h_{\id +\delta,F(\pi_0(A)\oplus M[-k-1])}F(\pi_0(A)).
$$

Since we know that $F(\pi_0(A))= \varinjlim F(\pi_0(A_{\alpha}))$, it suffices to show that for the images $\tilde{x}_{\alpha} \in F(\pi_0A_{\alpha}),\,x \in F(\pi_0A)$ of $\tilde{y}_{\alpha},\,y$, the maps
\[
\varinjlim F(\pi_0(A_{\alpha})\oplus M_{\alpha}[-k-1])\by_{ F(\pi_0(A_{\alpha}))}\{\tilde{x}_{\alpha}\}  \to F(\pi_0(A)\oplus M[-k-1])\by_{ F(\pi_0(A))}\{x\}
\]  
are equivalences. Taking homotopy groups, this becomes
$$
\varinjlim  \DD^{k+1-i}_{\tilde{x}_{\alpha}}(F, M_{\alpha})\to \DD^{k+1-i}_{x}(F, M),
$$
which by Lemma \ref{tantrans} is
$$
\varinjlim  \DD^{k+1-i}_{\tilde{x}}(F, M_{\alpha})\to \DD^{k+1-i}_{\tilde{x}}(F, M),
$$
for $\tilde{x} \in F(\pi_0A_{\beta})$ the image of $\tilde{y}$. 

It will therefore suffice to show that the functors $\DD^i_{\tilde{x}}(F, -): \Mod_{\pi_0A_{\beta}} \to \Ab$ preserve filtered colimits. If we express $\pi_0A_{\beta}$ as a filtered colimit of finitely generated $\pi_0R$-algebras, then the condition that $\pi^0F$ preserves filtered colimits allows us to write $[\tilde{x}]= [f_*z] \in \pi_0F(A)$, for $z \in F(C)_0$, with $C$ a finitely-generated $\pi_0R$-algebra. Then
$$
\DD^i_{\tilde{x}}(F, -)\cong \DD^i_{f_*z}(F, -)\cong\DD^i_{z}(F,f_* -),
$$
which preserves filtered colimits by hypothesis.
\end{proof}

\subsection{Sheaves}

\begin{definition}
Let $\oR\Tot_{\bS}: c\bS \to \bS$ be the derived total space functor from cosimplicial simplicial sets to simplicial sets, given by 
$$
\oR\Tot_{\bS}X^{\bt}= \holim_{\substack{\lla \\ n \in \Delta}} X^n,
$$
as in 
\cite{sht} \S VIII.1.  Explicitly, 
$$
\oR\Tot_{\bS} X^{\bt} =\{ x \in \prod_n (X^n)^{\Delta^n}\,:\, \pd^i_Xx_n = (\pd^i_{\Delta})^*x_{n+1},\,\sigma^i_Xx_n = (\sigma^i_{\Delta})^*x_{n-1}\}, 
$$ 
whenever $X$ is Reedy fibrant. Homotopy groups of the total space are related to a spectral sequence given in \cite{sht} \S VIII.1.
\end{definition}

\begin{definition}
A morphism  $f: A \to B$ in $s\Ring$  is said to be \'etale if $\pi_0f$ is \'etale and the
 maps $\pi_n(A)\ten_{\pi_0(A)}\pi_0(B) \to \pi_n(B)$ are isomorphisms for all $n$. An \'etale morphism  is said to be an \'etale covering if the morphism $\Spec \pi_0f: \Spec \pi_0B \to \Spec \pi_0A$ is a surjection of schemes.
\end{definition}

\begin{definition}\label{reedycofibrantdef}
Given $A \in s\Ring$ and  $B^{\bt} \in (s\Alg_A)^{\Delta}$, we may regard $B$ as a cocontinuous functor $B: \bS \to s\Alg_A$, determined by $B^n=B(\Delta^n)$. Then $ B^{\bt}$ is said to be Reedy cofibrant if the latching morphisms $f_n:B(\pd\Delta^n) \to B^n$ are cofibrations for all $n \ge 0$ (where $B(\pd\Delta^0)=  B(\emptyset)= A$).
\end{definition}

\begin{definition}\label{hypercoverdef}
A Reedy cofibrant object $B^{\bt} \in (s\Alg_A)^{\Delta}$ is  an \'etale hypercover if the latching morphisms are  \'etale coverings. An arbitrary object $C^{\bt} \in (s\Alg_A)^{\Delta}$ is an \'etale hypercover if there exists a levelwise weak equivalence $f: B^{\bt} \to C^{\bt}$, for $B^{\bt}$ a  Reedy cofibrant \'etale hypercover.
\end{definition}

\begin{definition}\label{cechdef}
Given a simplicial hypercover $A \to B^{\bt}$, and a presheaf $\sP$ over $A$, define the cosimplicial complex $\check{C}^{\bt}(B^{\bt}/A, \sP)$ by  $\check{C}^{n}(B^{\bt}/A, \sP) = \sP(B^n)$. 
\end{definition}

\begin{definition}\label{hypersheafdef}
A homotopy-preserving functor  $F: s\Alg \to \bS$  is said to be a hypersheaf for the \'etale topology if it satisfies the following conditions.
\begin{enumerate}
\item
It preserves finite products up to homotopy; this means that for any finite (possibly empty) subset $\{A_i\}$ of $s\Alg_R$, the map
$$
F(\prod A_i) \to \prod F(A_i)
$$
is a weak equivalence.

\item  For all \'etale hypercovers 
$A \to B^{\bt}$, the natural map
$$
F(A) \to \oR\Tot\check{C}(B^{\bt}/A , F)
$$
is a weak equivalence, for $\check{C}$ as in Definition \ref{cechdef}. 
\end{enumerate}
\end{definition}

\begin{remark}\label{stackhyper}
The same definition applies for functors $\Alg_{\pi_0R} \to \bS$. Given a groupoid-valued functor $\Gamma: \Alg_{\pi_0R} \to \gpd$, the nerve $B\Gamma : \Alg_{\pi_0R} \to \bS$ is a hypersheaf if and only if $\Gamma$ is a stack (in the sense of \cite{champs}).
\end{remark}

\begin{definition}\label{nilcompdef}
 Say that a  functor $F: s\Alg \to \b$ is nilcomplete if
for any $A \in  s\Alg_R$, the natural map $F(A) \to \Lim^h_kF(P_kA)$ to the homotopy limit is an equivalence.
\end{definition}

\begin{proposition}\label{sheafresult}
Take a homotopy-homogeneous nilcomplete homotopy-preserving  functor  $F: s\Alg \to \bS$. If
\begin{enumerate}
\item $\pi^0F: \Alg_{\pi_0R} \to \bS$ is a hypersheaf, and
\item for all  $A \in \Alg_{\pi_0R}$, all $x \in F(A)_0$, all $A$-modules $M$ and all \'etale morphisms $f:A \to A'$, the maps
\[
\DD_x^*(F, M)\ten_AA' \to \DD_{fx}^*(F, M\ten_AA')
\]
(induced by Lemma \ref{tantrans}) are isomorphisms,
\end{enumerate}
then $F$ is a hypersheaf.
\end{proposition}
\begin{proof}
Take an \'etale hypercover $f:A \to B^{\bt}$. 
The first observation to make is that $P_kA \to P_kB^{\bt}$ is also an \'etale hypercover. Assume inductively that
\[
F(P_{k-1}A) \to \oR\Tot\check{C}(P_{k-1}B^{\bt}/P_{k-1}A , F)
\]
is an equivalence (the case $k=1$ following because $\pi^0F$ is a hypersheaf). 
Now $P_{k}A \to P_{k-1}A$ is a square-zero extension up to homotopy  (see for instance \cite{hag2} Lemma 2.2.1.1), coming  from an essentially unique homotopy derivation $\delta: P_{k-1}A \to (\pi_kA)[-k-1]$, with
\[
P_kA \simeq  P_{k-1}A\by^h_{\id + \delta, \pi_0A \oplus (\pi_kA)[-k-1]} \pi_0A.
\]
Since $F$ is homotopy-homogeneous and homotopy-preserving, this means that
\[
F(P_kA) \simeq F(P_{k-1}A)\by^h_{F( \pi_0A \oplus (\pi_kA)[-k-1])} F(\pi_0A).
\]

For the inductive step, it suffices to show that for any point $x \in \pi^0F(A)$, the homotopy fibres of $F(P_kA)$ and of $ \oR\Tot\check{C}(P_kB^{\bt}/P_kA , F)$ over $x$ are weakly equivalent. From the expression above, we see that 
\[
F(P_kA)_x \simeq  F(P_{k-1}A)_x\by^h_{T_x(F/R, (\pi_kA)[-k-1])}\{0\},
\]
and the corresponding statement for the hypercover is
\begin{eqnarray*}
\oR\Tot\check{C}(P_{k}B^{\bt}/P_kA , F) &\simeq& \oR\Tot(\check{C}(P_{k-1}B^{\bt}/P_{k-1}A,F)_{fx}\by^h_{T_{fx}(F/R, (\pi_kB^{\bt})[-k-1])}\{0\} )\\
&\simeq& F(P_{k-1}A)_x\by^h_{ \oR\Tot T_{fx}(F/R, (\pi_kB^{\bt})[-k-1])}\{0\},
\end{eqnarray*}
using the inductive hypothesis and the fact the $\oR \Tot$ commutes with homotopy fibre products. 

This reduces the problem to showing that the map $T_x(F/R, (\pi_kA)[-k-1]) \to\oR\Tot T_{fx}(F/R, (\pi_kB^{\bt})[-k-1])  $ is a weak equivalence. Since the cohomology groups $\DD^*$ commute with \'etale base change, it follows that the map
\[
\oR\Tot T_{x}(F/R, (\pi_kA)[-k-1])\ten_AB^{\bt}\to \oR\Tot T_{fx}(F/R, (\pi_kB^{\bt})[-k-1])
\]
is a weak equivalence. Since $A \to B^{\bt}$ is an \'etale hypercover (and hence an fpqc hypercover), the map
\[
T_x(F/R, (\pi_kA)[-k-1]) \to \oR\Tot T_{x}(F/R, (\pi_kA)[-k-1])\ten_AB^{\bt}
\]
is also a weak equivalence, completing the inductive step.

Finally, since $F$ is nilcomplete, we get
\begin{eqnarray*}
F(A) &\simeq& {\Lim_k}^h F(P_kA)\\
\oR\Tot\check{C}(B^{\bt}/A , F) &\simeq& {\Lim_k}^h \oR\Tot\check{C}(P_{k}B^{\bt}/P_kA , F),
\end{eqnarray*}
which completes the proof.
\end{proof}

\subsection{Representability}

\begin{proposition} \label{cotexists}
Take a Noetherian simplicial ring $R$, and a homotopy-preserving functor $F:s\Alg_R \to \bS$, satisfying the following conditions:
\begin{enumerate}
 
\item For all discrete rings $A$, $F(A)$ is $n$-truncated, i.e. $\pi_iF(A)=0$ for all $i>n$ .

\item\label{cohesive} $F$ is homotopy-homogeneous, i.e. for all square-zero extensions $A \onto C$ and all maps $B \to C$, the map
$$
F(A\by_CB) \to F(A)\by_{F(C)}^hF(B)
$$
is an equivalence.

\item $F$ is nilcomplete, i.e. for all $A$, the map
$$
F(A) \to {\Lim}^h F(P_kA)
$$
is an equivalence, for $\{P_kA\}$ the Postnikov tower of $A$.

\item $F$ is a hypersheaf for the \'etale topology. 

\item $\pi^0F: \Alg_{\pi_0R} \to \bS$  preserves filtered colimits.

\item $R$ admits a dualising module, in the sense of \cite{lurie} Definition 3.6.1. Examples are anything  admitting a dualising complex in the sense of \cite{HaRD} Ch. V, such as $\Z$ or Gorenstein local rings, and any simplicial ring almost of finite presentation over a Noetherian ring with a dualising module.

\item for all finitely generated $A \in \Alg_{\pi_0R}$  and all $x \in F(A)_0$, the functors $\DD^i_x(F, -): \Mod_A \to \Ab$ preserve filtered colimits for all $i>0$.

\item for all finitely generated integral domains $A \in \Alg_{\pi_0R}$  and all $x \in F(A)_0$, the groups $\DD^i_x(F, A)$ are all  finitely generated $A$-modules.
\end{enumerate}

Then there is an almost perfect cotangent complex $\bL_{F/R}$ in the sense of \cite{lurie}. 
 \end{proposition}
\begin{proof}
This is an adaptation of  \cite{lurie} Theorem 7.4.1. After applying Theorem \ref{cohofp} to show that $F$ is almost of finite presentation,
the only difference is in condition (\ref{cohesive}), where we only consider square-zero extensions $A \to C$ (rather than all surjections), but also allow arbitrary maps $B \to C$ (rather than just surjections). The key observation is that we still satisfy the conditions of \cite{lurie} Theorem 3.6.9, guaranteeing local existence of the cotangent complex, while Lemma \ref{tantrans} provides the required compatibility.
\end{proof}

\begin{theorem}\label{lurierep}
Let $R$ be a derived G-ring admitting a dualising module, and   $F: s\Alg_R \to \bS$ a homotopy-preserving functor. Then $F$ is a geometric derived $n$-stack which is almost of finite presentation if and only if the conditions of Proposition \ref{cotexists} hold, and 
for all complete discrete local Noetherian  $\pi_0R$-algebras $A$, with maximal ideal $\m$, the map
$$
F(A) \to {\Lim}^h F(A/\m^r)
$$
is a weak equivalence.
\end{theorem}
\begin{proof}
This is essentially the same as \cite{lurie} Theorem 7.5.1, by combining ibid. Theorem 7.1.6 with Proposition \ref{cotexists} (rather than ibid. Theorem 7.4.1). 

Note that our revised condition (\ref{cohesive}) implies infinitesimal cohesiveness, since, for any square-zero extensions $0 \to M \to \tilde{A} \to A \to 0$, we may set $B$ to be the mapping cone (so $B \simeq A$), and consider the fibre product $\tilde{A} \simeq B\by^h_{A \oplus M[-1]}A$.  
 
To see that the revised condition (\ref{cohesive}) is necessary, we adapt \cite{lurie} Proposition 5.3.7. It suffices to show that for any smooth surjective map $U \to F$ of $n$-stacks, the map
$$
U(A)\by_{U(C)}^h U(B) \to F(A)\by_{F(C)}^h F(B)
$$ 
is surjective, for all square-zero extensions $A \onto C$. Moreover, the argument of \cite{lurie} Proposition 5.3.7 allows us to replace $A\by_CB$ with an \'etale algebra over it, giving a local lift of a point $x \in F(B)$ to $u \in U(B)$. The problem then reduces to showing that
$$
U(A)\by_{U(C)}^h U(B) \to F(A)\by_{F(C)}^h U(B)
$$ 
is surjective, but this follows from pulling back the surjection
$$
U(A) \to U(C)\by_{F(C)}^hF(A)
$$ 
given by the smoothness of $U \to F$.
\end{proof}

\begin{remark}\label{formalexistrk}
The Milnor exact sequence (\cite{sht} Proposition 2.15) gives a  sequence
\[
\bt \to \Lim^1_r \pi_{i+1}F(A/\m^r) \to \pi_i({\Lim}^h F(A/\m^r))  \to \Lim_r \pi_iF(A/\m^r)\to \bt,
\]
which is exact as groups for $i \ge 1$ and as pointed sets for $i=0$. Thus the condition of Theorem \ref{lurierep} can be rephrased to say that the map 
\[
f_0:\pi_0F(A) \to \Lim_r  \pi_0F(A/\m^r)
\] 
is surjective,  that for all $x \in F(A)$ the maps 
\[
f_{i,x}:\pi_i(FA,x) \to \Lim_r  \pi_i(F(A/\m^r),x)
\] 
are surjective for all $i \ge 1$ and that the resulting maps
\[
\ker f_{i,x} \to  \Lim^1_r \pi_{i+1}(F(A/\m^r),x)
\]
are surjective for all $i \ge 0$.

Now, we can say that an inverse system $\{G_r\}_{r \in \N}$ of groups satisfies the Mittag-Leffler condition if for all $r$, the images $\im(G_s \to G_r)_{s\ge r}$ satisfy the descending chain condition. In that case, the usual abelian proof (see e.g. \cite{W} Proposition 3.5.7) adapts to show that $\Lim^1 \{G_r\}_r = 1$. 

Hence, if each system $\{\im( \pi_{i}(F(A/\m^s),x)\to  \pi_{i}(F(A/\m^r),x))\}_{s \ge r}$ satisfies the Mittag-Leffler condition for $i \ge 1$, then the condition of Theorem \ref{lurierep} reduces to requiring that the maps
\[
\pi_iF(A) \to \Lim_r  \pi_iF(A/\m^r)
\]
be isomorphisms for all $i$.
\end{remark}

\begin{corollary}\label{lurierep2}
Let $R$ be a derived G-ring admitting a dualising module (in the sense of \cite{lurie} Definition 3.6.1) and   $F: s\Alg_R \to \bS$ a homotopy-preserving functor. Then $F$ is a geometric derived $n$-stack which is almost of finite presentation if and only if 
 the following conditions hold
\begin{enumerate}
 
\item For all discrete rings $A$, $F(A)$ is $n$-truncated, i.e. $\pi_iF(A)=0$ for all $i>n$ .

\item
$F$ is homotopy-homogeneous, i.e. for all square-zero extensions $A \onto C$ and all maps $B \to C$, the map
$$
F(A\by_CB) \to F(A)\by_{F(C)}^hF(B)
$$
is an equivalence.

\item $F$ is nilcomplete, i.e. for all $A$, the map
$$
F(A) \to {\Lim}^h F(P_kA)
$$
is an equivalence, for $\{P_kA\}$ the Postnikov tower of $A$.

\item $\pi^0F:\Alg_{\pi_0R} \to \bS$ is a hypersheaf for the \'etale topology. 

\item\label{colim1} $\pi_0\pi^0F:  \Alg_{\pi_0R} \to \Set$  preserves filtered colimits.

\item\label{colim2} For all $A \in \Alg_{\pi_0R}$ and all $x \in F(A)$, the functors $\pi_i(\pi^0F,x): \Alg_A \to \Set$  preserve filtered colimits for all $i>0$.

\item \label{shf2}
for all finitely generated integral domains $A \in \Alg_{\pi_0R}$, all $x \in F(A)_0$ and all \'etale morphisms $f:A \to A'$, the maps
\[
\DD_x^*(F, A)\ten_AA' \to \DD_{fx}^*(F, A')
\]
are isomorphisms.

\item for all finitely generated $A \in \Alg_{\pi_0R}$  and all $x \in F(A)_0$, the functors $\DD^i_x(F, -): \Mod_A \to \Ab$ preserve filtered colimits for all $i>0$.

\item for all finitely generated integral domains $A \in \Alg_{\pi_0R}$  and all $x \in F(A)_0$, the groups $\DD^i_x(F, A)$ are all  finitely generated $A$-modules.

\item for all complete discrete local Noetherian  $\pi_0R$-algebras $A$, with maximal ideal $\m$, the map
$$
F(A) \to {\Lim_n}^h F(A/\m^r)
$$
is a weak equivalence (see Remark \ref{formalexistrk} for a reformulation).
\end{enumerate}
\end{corollary}
\begin{proof}
If $F$ is a derived $n$-stack of almost finite presentation, then the \'etale sheaf $A' \mapsto \DD_{fx}^i(F, A')$ on $Y:=\Spec A$ is just
\[
\ext^i_{\O_Y}(x^*\bL^{F/R}, \O_Y), 
\]
which is necessarily quasi-coherent, as $x^*\bL^{F/R} $ is equivalent to a complex of finitely generated locally free sheaves (for instance by the results of \cite{stacks2} \S \ref{stacks-qucohsn}). Combined with Theorem \ref{lurierep}, this ensures that all the conditions are necessary, once we note that  conditions \ref{colim1} and \ref{colim2} are equivalent to $\pi^0F: \Alg_{\pi_0R} \to \bS$  preserving filtered colimits.

For the converse, we just need to show that $F$ is a hypersheaf in order to ensure that it satisfies the conditions of Theorem \ref{lurierep}. This follows almost immediately from Proposition \ref{sheafresult}, first noting that condition (\ref{shf2}) above combines with almost finite presentation and exactness of the tangent complex to ensure that for all
$A \in \Alg_{\pi_0R}$, all $x \in F(A)_0$, all $A$-modules $M$ and all \'etale morphisms $f:A \to A'$, the maps
\[
\DD_x^*(F, M)\ten_AA' \to \DD_{fx}^*(F, M\ten_AA')
\]
are isomorphisms.
 \end{proof}

\begin{remark}
Although Corollary \ref{lurierep2} seems more complicated than Theorem \ref{lurierep}, since it has an extra condition, it is much easier to verify in practice. This is because $F(A)$ is only $n$-truncated when $A$ is discrete, so it is much easier to check that $\pi^0F$ is a hypersheaf than to do the same for $F$. 
\end{remark}

\begin{proposition}\label{detectweak}
Take a morphism $\alpha: F \to G$ of  almost  finitely presented geometric derived $n$-stacks a over $R$. Then $\alpha$ is a weak equivalence if and only if 
\begin{enumerate}
        \item $\pi^0\alpha: \pi^0F \to \pi^0G$ is a weak equivalence of functors $\Alg_{\pi_0R}\to \bS$, and
\item the maps $\DD^i_x(F,A) \to \DD^i_{\alpha x}(G,A)$ are isomorphisms for all  finitely generated integral domains $A \in \Alg_{\pi_0R}$, all $x \in F(A)_0$, and all $i>0$.
\end{enumerate}
\end{proposition}
\begin{proof}
 It suffices to show that $\bL_{\bt}^{F/G} \simeq 0$. For if this is the case, then \cite{hag2} Corollary 2.2.5.6 implies that $\alpha$ is \'etale. By applying \cite{hag2} Theorem 2.2.2.6 locally, it follows that an \'etale morphism $\alpha$ must be a  weak equivalence whenever $\pi^0\alpha$ is so.

Now, $  \bL_{\bt}^{F/G}$ is the cone of $\alpha^*\bL^{G/R}\to \bL^{F/R}$, so we wish  to show that this map is an equivalence locally. This is equivalent to saying that for all integral domains $A \in \pi_0R$, all $\pi_0R$-modules $M$, all $x \in F(A)$ and all $i$, the maps
\[
 \DD^i_x(F,M) \to \DD^i_{\alpha x}(G,M)       
\]
  are isomorphisms.

For $i\le 0$, these isomorphisms follow immediately from the hypothesis that $\pi^0\alpha$ be an equivalence. For $i>0$, we first note that finite presentation of $\pi^0F$ means that we may assume that $A$  is finitely generated. We then have an almost perfect complex $x^*\bL_{\bt}^{F/G}$ with the property that 
\[
 \Ext_{A}^i(  x^*\bL_{\bt}^{F/G},A)=0     
\]
for all $i$, so $\Ext_{A}^i(  x^*\bL_{\bt}^{F/G},P)=0$ for all almost perfect $A$-complexes $P$ (using nilcompleteness of $F$ and $G$). In particular, 
\[
  \DD^i_x(F/G, M)= \Ext_{A}^i(  x^*\bL_{\bt}^{F/G},M)=0       
\]
for all finite $A$-modules. Almost finite presentation of $F$ and $G$ now gives that $\DD^i_x(F/G, M)=0$ for all $A$-modules $M$, completing the proof.
\end{proof}

\subsection{Strong quasi-compactness}

\begin{lemma}\label{prodfields}
If $S$ is a set of separably closed fields, and $X=\Spec (\prod_{k \in S}k)$, then every surjective \'etale morphism  $f:Y \to X$ of affine schemes has a section.
\end{lemma}
\begin{proof}
Since 
$f$ is surjective, the canonical maps $\Spec k \to X$ admit lifts to $Y$, for all $k \in S$, combining to give  a map $\coprod_{k \in S} \Spec k \to Y$. Since $Y$ is affine, this is equivalent to giving a map $X \to Y$, and this is automatically a section of $f$.
\end{proof}

\begin{proposition}\label{sqc}
A morphism $F \to G$  of geometric $m$-stacks is strongly quasi-compact if and only if for all sets $S$  of separably closed fields, the map
$$
F(\prod_{k \in S} k) \to (\prod_{k \in S} F(k))\by_{(\prod_{k \in S} G(k))}^hG(\prod_{k \in S} k)
$$
is a weak equivalence in $\bS$.
\end{proposition}
\begin{proof}
Let $Z= \Spec (\prod_{k \in S} k)$, and fix an element $g \in G(Z)$ . If  $F \to G$ is strongly quasi-compact, then $F\by_{G,g}^hZ$ is strongly quasi-compact, so by \cite{stacks2} Theorem \ref{stacks-relstrict}, there exists a simplicial affine scheme $X$ whose sheafification $X^{\sharp}$ is $F\by^h_GZ$. Now, Lemma \ref{prodfields} implies that $Z$ is weakly initial in the category of \'etale hypercovers of $Z$, so (for instance by \cite{stacks2} Corollary \ref{stacks-duskinmor}) $X^{\sharp}(Z) \simeq X(Z)$. Now, since $X$ is simplicial affine, it preserves arbitrary limits of rings, so
$$
X(Z)\cong \prod_{k \in S}X(k) \cong \prod_{k \in S}X^{\sharp}(Z), 
$$   
which proves that the condition is necessary.

To prove that the condition is sufficient, we need to show that for any affine scheme $U$ and any morphism $U \to G$, the homotopy fibre product $F\by^h_GU$ is  strongly quasi-compact. Since  $U$ is affine, it satisfies the condition, so $F\by^h_GU$ will also, and so we may assume that  $G =U$ or even $ \Spec \Z$.

Now, it follows (for instance from the proof of \cite{stacks2} Theorem \ref{stacks-relstrict}) that if an $n$-geometric stack $F$ admits an $n$-atlas $U\to F$, with $U$ quasi-compact, and the diagonal $F \to F\by F$ is strongly quasi-compact, then $F$ is strongly quasi-compact.

We will proceed by induction on $n$ (noting that we use $n$-geometric, as in Remark \ref{cflurie}, rather than $n$-truncated). A $0$-geometric stack $F$ is a disjoint union of affine schemes, so is separated, and in particular its diagonal is strongly quasi-compact. 

Assume that an $n$-geometric stack $F$ has strongly quasi-compact diagonal and satisfies the condition above, and take an $n$-atlas $V \to F$ for $V$ $0$-geometric (where we interpret a $0$-atlas as an isomorphism).   Let $S$ be a set of representatives of equivalence  classes of geometric points of $V$, and set $Z=\Spec (\prod_{k \in S} k)$. Since $F$ satisfies the condition above, 
$$
F(Z) \cong \prod_{k \in S} F(k),
$$ 
so the points $\Spec k \to V\to F$ combine to define a map $f:Z \to F$. 

As $V \to F$ is an atlas, for some \'etale cover $Z'\to Z$, $f$ lifts to a map $\tilde{f}: Z' \to V$. But  Lemma \ref{prodfields} implies that $Z' \to Z$ has a section, so we have a lifting $\tilde{f}: Z \to V$ of $f$. Now, $V= \coprod_{\alpha \in I} V_{\alpha}$ is a disjoint union of affine schemes, and since $Z$ is  quasi-compact, there is some  finite subset $J \subset I$ with $U:=\coprod_{\alpha \in J} V_{\alpha}$ containing the image of $Z$. But $U$ is then quasi-compact, and $U \to F$ is surjective, hence an $n$-atlas, which completes the induction.
\end{proof}

\begin{corollary}
A morphism $F \to G$  of geometric derived  stacks is strongly quasi-compact if and only if for all sets $S$  of separably closed fields, the map
$$
F(\prod_{k \in S} k) \to (\prod_{k \in S} F(k))\by_{(\prod_{k \in S} G(k))}^hG(\prod_{k \in S} k)
$$
is a weak equivalence in $\bS$.
\end{corollary}
\begin{proof}
The morphism $F \to G$ is strongly quasi-compact if and only if $\pi^0F \to \pi^0G$ is a strongly quasi-compact morphism of geometric  stacks, so we apply Proposition \ref{sqc}.
\end{proof}

\section{Complete simplicial and chain algebras}\label{completesn}

\begin{proposition}\label{adjointmodel}
Take a cofibrantly generated model category $\C$. Assume that $\cD$ is a complete and cocomplete category, equipped with  an adjunction
\[
\xymatrix@1{\cD \ar@<1ex>[r]^U_{\top} & \C \ar@<1ex>[l]^F},
\]
with $U$ preserving filtered colimits. If $UF$ maps generating trivial cofibrations to weak equivalences, 
then $\cD$ has a cofibrantly generated model structure in with a morphism $f$ is a fibration or a weak equivalence whenever $Uf$ is so. 

This adjunction is a pair of Quillen equivalences if and only if  the unit morphism $A \to UFA$ is a weak equivalence for all cofibrant objects $A \in \C$.
\end{proposition}
\begin{proof}
To see that this defines a model structure on $\cD$, note that since $U$ preserves filtered  colimits, for any small object $I \in \C$, the object $FI$ is small in $\cD$, so we may apply \cite{Hirschhorn} Theorem 11.3.2 to obtain the model structure on $\cD$.

Since $U$ reflects weak equivalences, by \cite{Hovey} Corollary 1.3.16, 
the functors $F \vdash U$ form a pair of Quillen equivalences if and only if  the morphisms $\oR \eta: A \to \oR U FA$  
are  
weak equivalences for all cofibrant $A \in \C$.
Since $U$ preserves weak equivalences, the map $UB \to \oR UB$ is a weak equivalence for all $B \in \cD$.
Thus the unit $\eta: A\to UFA$ is a weak equivalence if and only if  $\oR \eta$ is so. 
\end{proof}

Fix a Noetherian ring $R$.

\begin{definition}
Say that a simplicial $R$-algebra $A$ is finitely generated if there are finite sets $\Sigma_q \subset A_q$ of generators, closed under the degeneracy operations, with only finitely many elements of $\bigcup_q \Sigma_q$ being non-degenerate. 

Define $FG s\Alg_R$ to be the category of finitely generated simplicial $R$-algebras.  Define $FG dg_+\Alg_R$ to be the category of finitely generated non-negatively graded chain $R$-algebras (if $R$ is a $\Q$-algebra).
\end{definition}

\begin{definition}
Given $A \in s\Alg_R$, define $\hat{A}:=\Lim_n A/I^n_A$, for   $I_A = \ker (A \to \pi_0A)$. Given $A \in dg_+\Alg_R$, define $\hat{A}:=\Lim_n A/I^n_A$,  for  $I_A = \ker (A \to \H_0A)$.
\end{definition}

\begin{definition}
Define  $\widehat{FG s\Alg_R}$ to be the full subcategory of $s\Alg_R$ consisting of objects of the form $\hat{A}$, for $A \in FG s\Alg_R$. Define  $\widehat{FG dg_+\Alg_R}$ to be the full subcategory of $dg_+\Alg_R$ consisting of objects of the form  $\hat{A}$, for $A \in FG dg_+\Alg_R$
\end{definition}

\begin{lemma}\label{indfinitecolim}
The  categories $\widehat{FG s\Alg_R}$ and $\widehat{FG dg_+\Alg_R}$ contain all finite colimits.
\end{lemma}
\begin{proof}
 The initial object is $\hat{R}$ (which equals $R$ whenever $R$ is discrete), and the cofibre coproduct  of $A \la B \to C$ is given by 
\[
 A\hat{\ten}_BC:= \widehat{A\ten_BC}.
\]
\end{proof}

\begin{proposition}\label{indcolex}
 For  $\C=\widehat{FG s\Alg_R}$ or $\widehat{FG dg_+\Alg_R}$, the category $\ind(\C)$  is equivalent to the category of left-exact functors  $F:\C^{\op}\to \Set$, i.e. functors for which 
\begin{enumerate}
 
\item $F(\hat{R})$ is the one-point set, and

\item the map
\[
 F(A\hat{\ten}_BC) \to F(A)\by_{F(B)}F(C)
\]
is an isomorphism for all diagrams $A \la B \to C$.
\end{enumerate}

The equivalence is given by sending a direct system $\{A_{\alpha}\}_{\alpha}$ to the functor $F(B) = \LLim_{\alpha} \Hom_{\C}(B,A_{\alpha})$.
\end{proposition}
\begin{proof}
For $A \in \C$, a subobject of $A^{\op} \in \C^{\op}$ is just a surjective map $A \to B$ in $\C$, or equivalently a simplicial (resp. dg) ideal of $A$. Since $A$ is Noetherian, it satisfies ACC on such ideals, and hence $A^{\op}$ satisfies DCC on strict subobjects. Therefore $\C^{\op}$ is an Artinian category containing all finite limits, so the required  result is given by \cite{descent}, Corollary to Proposition 3.1.
\end{proof}

\begin{proposition}\label{cNhat}
There are cofibrantly generated model structures on the categories $\ind(\widehat{FGs\Alg_R})$ and $\ind(\widehat{FGdg_+\Alg_R})$ in which a morphism $f: \{A_{\alpha}\}_{\alpha} \to  \{B_{\beta}\}_{\beta}$ is a fibration or a weak equivalence whenever the corresponding map
\[
\LLim f: \LLim_{\alpha}A_{\alpha}\to \LLim_{\beta}B_{\beta}
\]
in $s\Alg_R $ or $dg_+\Alg_R$ is so. 

For these model structures, the functors
\begin{eqnarray*}
U:\ind(\widehat{FGs\Alg_R})&\to& s\Alg_R\\
U:\ind(\widehat{FGdg_+\Alg_R})&\to& dg_+\Alg_R
\end{eqnarray*}
given by $U(\{A_{\alpha}\}_{\alpha})= \LLim_{\alpha}A_{\alpha}$ are right Quillen equivalences.
\end{proposition}
\begin{proof}
We begin by showing that $\ind(\widehat{FGs\Alg_R})$ and $\ind(\widehat{FGdg_+\Alg_R})$ are complete and cocomplete. By Lemma \ref{indfinitecolim}, they contain finite colimits, and the proof of \cite{isaksen} Proposition 11.1 then  ensures that they contain arbitrary coproducts, and hence arbitrary colimits. It follows immediately from Proposition \ref{indcolex} that the categories contain arbitrary limits, since any limit of left-exact functors is left-exact. 

We  need to establish that the functors $U$ have left adjoints. Since $R$ is Noetherian, finitely generated objects over $R$ are finitely presented, so the functors
\begin{eqnarray*}
\LLim:\ind(FGs\Alg_R)&\to& s\Alg_R\\
\LLim:\ind(FGdg_+\Alg_R)&\to& dg_+\Alg_R
\end{eqnarray*}
are equivalences of categories. The left adjoints
 \begin{eqnarray*}
F:\ind(FGs\Alg_R)&\to& \ind(\widehat{FGs\Alg_R})\\
F:\ind(FGdg_+\Alg_R)&\to&\ind(\widehat{FGdg_+\Alg_R}) 
\end{eqnarray*}
to $U$ are thus given by $\{A_{\alpha}\}_{\alpha}\mapsto \{\hat{A}_{\alpha}\}_{\alpha}$.

It is immediate that $U$ preserves filtered colimits, so we may apply Proposition \ref{adjointmodel} to construct the model structures. It only remains to show that $U$ is a Quillen equivalence. By Proposition \ref{adjointmodel}, we need only show that, for any cofibrant $A \in s\Alg_R$ or $A \in dg_+\Alg_R$, the map
\[
A \to UFA 
\]
is a weak equivalence. If we write $A= \LLim_{\alpha} A_{\alpha}$, for $A_{\alpha} \in FGs\Alg_R$ (or $A_{\alpha} \in FGdg_+\Alg_R$), then
\[
UFA= \LLim_{\alpha} \hat{A}_{\alpha}.
\]

Thus it suffices to show that for $A \in FGs\Alg_R$ (or $A \in FGdg_+\Alg_R$), the map $A \to \hat{A}$ is a weak equivalence. If $A \in FGs\Alg_R$, then each $A_n$ is Noetherian, so \cite{stacks2} Theorem \ref{stacks-fthm} gives the required equivalence. If $A \in FGdg_+\Alg_R$, then $A_0$ is Noetherian and each $A_n$ is a finite $A_0$-module, so \cite{stacks2} Lemma \ref{stacks-dgshrink} gives the required equivalence.
\end{proof}

\begin{lemma}\label{hatful}
The  category $\ind(\widehat{FGs\Alg_R})$ (resp. $\ind(\widehat{FGdg_+\Alg_R})$ ) is equivalent to a full subcategory $\C$ of $s\Alg_R$ (resp. $dg_+\Alg_R$). If $I_A= \ker(A \to \H_0A)$ , then $A$ is an object of $\C$ if and only if it contains the $I_A$-adic completions of all its finitely generated subalgebras.
\end{lemma}
\begin{proof}
It is immediate that $A$ satisfies the condition above if and only if $A=UFA$ for the functors $U$ and $F$ from the proof of Proposition \ref{cNhat}.
Thus we need only show that the functor $U:\ind(\widehat{FGs\Alg_R}) \to Fs\Alg $ given by $\{A_{\alpha}\} \mapsto \LLim_{\alpha} A_{\alpha}$ is full and faithful. It suffices to show that for $A \in \widehat{FGs\Alg_R}$ and $B \in \ind(\widehat{FGs\Alg_R})$, $\Hom(A, \LLim B_{\beta})= \LLim_{\beta}  \Hom(A,B_{\beta})$. 

To do this, recall that $A =\widehat{A'}$ for some finitely generated $A'$, and express $A$ as $\LLim A_{\alpha}$, for $A' \subset A_{\alpha} \in FGs\Alg_R$. Then
\begin{eqnarray*}
\Hom(A, \LLim B_{\beta}) &=& \Lim_{\alpha} \Hom(A_{\alpha}, \LLim B_{\beta})\\
&=& \Lim_{\alpha} \LLim_{\beta} \Hom(A_{\alpha},  B_{\beta})\\
&=& \Lim_{\alpha} \LLim_{\beta}\Hom(\hat{A}_{\alpha}, B_{\beta}),
\end{eqnarray*}
but $\hat{A}_{\alpha}= A$, giving the required result.
\end{proof}

\subsection{Nilpotent algebras}


\begin{definition}\label{littledef}
Say that a surjection $A \to B$ in  
$dg_+\Alg_R$ (resp. $s\Alg_R$)
is a \emph{little extension} if the kernel $K$ satisfies $I_A\cdot K=0$. Say that an acyclic little extension is \emph{tiny} if $K$ (resp. $NK$) is of the form $\cone(M)[-r]$ for some $\H_0A$-module  (resp. $\pi_0A$-module) $M$.
\end{definition}

Note that acyclic little extensions are necessarily square-zero, but that arbitrary little extensions need not be.

\begin{definition}
Define $dg_+\cN_R$ (resp. $s\cN_R$)  to be the full subcategory of $dg_+\Alg_R$ (resp. $s\Alg_R$) consisting of objects $A$
 for which the map $A \to \H_0A$ (resp. $A \to \pi_0A$) has nilpotent kernel. Define $dg_+\cN_R^{\flat}$ (resp. $s\cN_R^{\flat}$) to be the full subcategory of $dg_+\cN_R$ (resp. $s\cN_R$) consisting of objects $A$
 for which $A_i=0$ (resp. $N_iA=0$) for all $i \gg 0$.
\end{definition}

\begin{lemma}\label{wtiny}
Every  surjective weak equivalence $f:A \to B$ in $dg_+\cN_R^{\flat}$ (resp. $s\cN_R^{\flat}$) factors as a composition of tiny acyclic  extensions.
\end{lemma}
\begin{proof}
We first prove this for $dg_+\cN_R^{\flat}$. Let $K= \ker(f)$, and observe that the good truncations
\[
(\tau_{\ge r}K)_i= \left\{ \begin{matrix} K_i & i>r \\ \z_r K & i=r \\ 0 & i<r \end{matrix} \right.  
\]
are also dg ideals in $A$. Since $A$ is concentrated in degrees $[0,d]$ for some $d$, we get a factorisation of $f$ into acyclic surjections
\[
A= A/(\tau_{\ge d}K) \to A/(\tau_{\ge (d-1)}K)\to \ldots \to A/(\tau_{\ge 0}K)=B.
\]
We therefore reduce to the case where $K$ is concentrated in degrees $r,r+1$.

Let $s$ be least such that $K_r \cdot I_A^s=0$; if $s=1$ then $f$ is already a tiny acyclic extension.
We will proceed by induction on $s$. Since $K \onto (K/I_A K)$, we have $\H_r(K/I_AK )=0$. This means that the inclusion $\tau_{> r}( K/I_A K ) \to (K/I_A K )$ is a quasi-isomorphism of ideals in $A$. 
If we set $B':= (A/I_AK)/( \tau_{>r}K/I_AK)$ and $K'':=\ker(A \to B')$, then  $I_A\cdot K''=0$ so   $f'':B' \to B$ is an acyclic little extension. In fact, for $M:=(K/I_A\cdot K )_r$, we have $K'' = \cone(M)[-r]$, so $f''$ is a tiny acyclic extension.

Now, for $K':=\ker(f':A \to B')$ we have $K'_r= (I_AK)_r$, so $ K_r' \cdot I_A^{s-1}=0$, so by induction  $f'$ factors as a composition of tiny acyclic  extensions. This completes the inductive step.

Finally, for $f:A \to B$ in $s\cN_R^{\flat}  $ , normalisation gives an equivalence of categories between simplicial $A$-modules and non-negatively graded dg $NA$-modules. In particular, it gives an equivalence between the categories of ideals, and hence  quotients of $A$ correspond to quotients of $NA$. If $Nf$ is a tiny acyclic extension, then so is $f$,  since $NK$ is automatically an $\H_0NA$-module, and $\H_0NA= \pi_0A$. The proof above expresses
$NA \to NB$ as a composition of tiny acyclic extensions, which thus yields such an expression for $A\to B$. 
\end{proof}

\begin{definition}
 Define $\widehat{FGs\Alg_R}^{\flat}$ (resp. $\widehat{FGdg_+\Alg_R}^{\flat}$) to be the full subcategory of $\widehat{FGs\Alg_R}$ (resp. $\widehat{FGdg_+\Alg_R}$) consisting of objects $A$
 for which $A_i=0$ (resp. $N_iA=0$) for all $i \gg 0$.
\end{definition}

\begin{lemma}\label{wsmall}
For any surjective weak equivalence $f:A \to B$ in $\widehat{FGs\Alg_R}^{\flat}$ (resp. $\widehat{FGdg_+\Alg_R}^{\flat}$), the associated morphism
\[
\{A/I_A^n\} \to \{B/I_B^n\}
\]
in $\pro(dg_+\cN_R^{\flat})$ (resp. $\pro(s\cN_R^{\flat})$) is isomorphic to an inverse limit of surjective weak equivalences in $dg_+\cN_R$ (resp. $s\cN_R$).
\end{lemma}
\begin{proof}
With reasoning as at the end of Lemma \ref{wtiny}, it suffices to prove this for $\widehat{FGdg_+\Alg_R}^{\flat}$.
The first observation to make is that if $f$ and $g$ are composable morphisms satisfying the conclusions of this lemma, then $fg$ also satisfies the conclusions. 
Let $K= \ker(f)$; since $A$ is concentrated in degrees $[0,d]$ for some $d$, we get a factorisation of $f$ into acyclic surjections
\[
A= A/(\tau_{\ge d}K) \to A/(\tau_{\ge (d-1)}K)\to \ldots \to A/(\tau_{\ge 0}K)=B,
\]
and therefore reduce to the case where $K$ is concentrated in degrees $r,r+1$.

Set $I:= I_A$ and $J:= I_B$;
we now define a dg ideal $I(n)' \lhd A$ to be generated by $I^n$ and $K_{r+1}\cap d^{-1}(I^n) $, and set $A(n)':= A/ I(n)'$. There is a surjection $A(n)' \to B/J^n$, with kernel $K/ (K\cap I(n)')$. This is given by
\[
(K/ K\cap I(n)')_i= \left\{ \begin{matrix} K_r/(K \cap I^n)_r & i=r\\  K_{r+1}/ (K_{r+1}\cap d^{-1}(I^n)) & i = r+1 \\ 0 & i \ne r, r+1. \end{matrix} \right.
\]
Since $d:K_{r+1} \to K_r$ is an isomorphism, so is $d: K_{r+1} \cap d^{-1} I^n \to (K\cap I^n)_r$, which  means that $\H_*(K/ K\cap I(n)' )=0$, so $A(n') \to B/J^n$ is a
 weak equivalence.

Thus it only remains to show that the pro-objects $\{A/ I^n\}_n$ and $\{A/ I(n)'\}_n$ are isomorphic. Since $I^n \subset I(n)'$, there is an obvious morphism $A/ I^n \to A/ I(n)'$, and it remains to construct an inverse in the pro-category,  Observe that $A_0$ is a Noetherian ring, and that $(I^r)_r$ and $K_r$ are finitely generated $A_0$-modules.

Now,  $(K\cap I^n)_r= K_r \cap I_0^{n-r} (I^r)_r$ for all $n \ge r$. By the Artin--Rees Lemma (\cite{Mat} Theorem 8.5),  there exists some $c\ge r$ such that for all $n \ge c$, this is 
\[
I_0^{n-c}(K_r \cap I_0^{c-r} (I^r)_r)= I_0^{n-c}(K_r \cap (I^c)_r).
\] 
Thus $K_{r+1} \cap d^{-1}(I^n)$  is just $ I_0^{n-c}K_{r+1} \cap d^{-1}(I^c)$. Therefore $I(n)' \subset I^{n-c}$, so  giving maps $A/ I(n)' \to A/I^{n-c}$, and hence the required inverse in the pro-category.
\end{proof}

\subsection{A nilpotent representability theorem}

 Let $d\cN_R^{\flat}$ (or simply $d\cN^{\flat}$) be either of the categories $s\cN_R^{\flat}$ or $dg_+\cN_R^{\flat}$.

\begin{remark}
Note that the constructions of \S \ref{tgtsn} carry over to the categories $d\cN_R^{\flat}$, since they are closed under fibre products.
\end{remark}

\begin{lemma}\label{properw}
Given a weak equivalence $f:A \to B$ between fibrant objects in a right proper model category $\C$, there exists a diagram
$$
\xymatrix{ & & B\\
A \ar[r]^i & C \ar[ur]^{g_1}\ar[dr]_{g_0}\\
& & A},
$$
such that $g_0, g_1$ are trivial fibrations, $g_1 \circ i =f$ and $g_0 \circ i = \id$.
 \end{lemma}
\begin{proof}
Let $C:= A\by_{f, B, \ev_0}B^I$, for $B^I$ the path object of $B$, and let $g_0$ be given by projection onto $A$. The projection $C \to B^I$ is the pullback of $A \to B$ along the fibration $B^I \to B$, so is a weak equivalence by right properness. Define $g_1$ to be the composition of this with the trivial fibration $\ev_1:B^I \to B$. The projection $g_0$ is the pullback of the trivial fibration $\ev_0:B^I \to B$ along $f$, so is a trivial fibration.  

It only remains  to show that $g_1$ is a fibration. Since $B^I \to B\by B$ is a fibration, pulling back along $f$ shows that $(g_0,g_1):C \to A \by B$ is a fibration, and since $A$ is fibrant, we deduce that $A\by B \to B$ is a fibration, so $g_1$ must be a fibration.
\end{proof}

\begin{lemma}\label{cohofp2}
If a homotopy-preserving functor $F: d\cN_R^{\flat} \to \bS$  is homotopy-homogeneous, then it is almost of finite presentation if and only if the following hold:
\begin{enumerate}
\item the functor $\pi^0F: \Alg_{\pi_0R} \to \bS$  preserves filtered colimits;
 
\item for all finitely generated $A \in \Alg_{\pi_0R}$  and all $x \in F(A)_0$, the functors $\DD^i_x(F, -): \Mod_A \to \Ab$ preserve filtered colimits for all $i>0$.
\end{enumerate}
\end{lemma}
\begin{proof}
This is essentially the same as Theorem \ref{cohofp} --- we need only show that any square-zero extension $A \to B$ in $ s\cN_R^{\flat}$ (resp. $ dg_+\cN_R^{\flat} $) is of the form $A = B\by_{B \oplus M}\tilde{B}$, for $\tilde{B} \to B$ a weak equivalence, and some derivation $B \to M$. Now just note that such an expression is constructed in the proof of Lemma \ref{luriesmall}.
\end{proof}

\begin{theorem}\label{lurierep3}
Let $R$ be a derived G-ring admitting a dualising module (in the sense of \cite{lurie} Definition 3.6.1) and   take a functor $F: d\cN_R^{\flat} \to \bS$. Then $F$ is the restriction of an almost finitely presented  geometric derived $n$-stack $F':d\Alg_R \to \bS$ if and only if 
 the following conditions hold

\begin{enumerate}
 
\item $F$ maps tiny acyclic extensions to weak equivalences.

\item For all discrete rings $A$, $F(A)$ is $n$-truncated, i.e. $\pi_iF(A)=0$ for all $i>n$ .

\item
$F$ is homotopy-homogeneous, i.e. for all square-zero extensions $A \onto C$ and all maps $B \to C$, the map
$$
F(A\by_CB) \to F(A)\by_{F(C)}^hF(B)
$$
is an equivalence.

\item $\pi^0F:\Alg_{\pi_0R} \to \bS$ is a hypersheaf for the \'etale topology. 

\item\label{afp1a} $\pi_0\pi^0F:  \Alg_{\pi_0R} \to \Set$  preserves filtered colimits.

\item\label{afp1b} For all $A \in \Alg_{\pi_0R}$ and all $x \in F(A)$, the functors $\pi_i(\pi^0F,x): \Alg_A \to \Set$  preserve filtered colimits for all $i>0$.

\item 
for all finitely generated integral domains $A \in \Alg_{\pi_0R}$, all $x \in F(A)_0$ and all \'etale morphisms $f:A \to A'$, the maps
\[
\DD_x^*(F, A)\ten_AA' \to \DD_{fx}^*(F, A')
\]
are isomorphisms.

\item\label{afp2} for all finitely generated $A \in \Alg_{\pi_0R}$  and all $x \in F(A)_0$, the functors $\DD^i_x(F, -): \Mod_A \to \Ab$ preserve filtered colimits for all $i>0$.

\item for all finitely generated integral domains $A \in \Alg_{\pi_0R}$  and all $x \in F(A)_0$, the groups $\DD^i_x(F, A)$ are all finitely generated $A$-modules.

\item for all complete discrete local Noetherian  $\pi_0R$-algebras $A$, with maximal ideal $\m$, the map
$$
\pi^0F(A) \to {\Lim}^h F(A/\m^r)
$$
is a weak equivalence (see Remark \ref{formalexistrk} for a reformulation).
\end{enumerate}
Moreover, $F'$ is uniquely determined by $F$ (up to weak equivalence).
\end{theorem}
\begin{proof}
We will deal with the simplicial case. Since normalisation gives an equivalence $N: s\cN_R^{\flat} \to dg_+\cN_R^{\flat}$ when $R$ is a $\Q$-algebra, the dg case is entirely similar.

First observe that $F$ extends to a functor $\hat{F}:\pro(s\cN_R^{\flat}) \to \bS $, given by $\hat{F}(\{A^{(i)}\}_{i \in I})= \Lim^h_{i \in I} F(A^{(i)})$.  

Define $F'$ as follows. For any $A \in s\Alg_R$, write $A = \LLim A_{\alpha}$, for $A_{\alpha} \in FGs\Alg_R$, and set
\[
F'(A) := {\Lim_k}^h \LLim_{\alpha} \hat{F}( \{P_kA_{\alpha}/I_{A_{\alpha}}^n\}_{n \in \N}).
\]

We first show that $F'$ is homotopy-preserving; it follows from Lemma \ref{wtiny} and the proof of Proposition \ref{cNhat}  that $F$ is homotopy-preserving.
Note that the formula for  $F'$  defines a functor $F''$ on $\ind(\widehat{FGs\Alg_R})$, and that $F'$ is the composition of $F''$ with the derived left Quillen functor of Proposition \ref{cNhat}.  
By the proof of Lemma \ref{properw}, it suffices to show that $F''$ maps trivial fibrations to weak equivalences. Any such morphism is isomorphic to one of the form $\{A_{\alpha}\}_{\alpha}\to \{B_{\alpha}\}_{\alpha}$, where each $A_{\alpha} \to B_{\alpha}$ is a surjective weak equivalence in $\widehat{FGs\Alg_R}$. Note that $P_kA_{\alpha} \to P_kB_{\alpha}$ is also a surjective weak equivalence, so we may apply Lemma \ref{wsmall}, which implies that
\[
\hat{F}( \{P_kA_{\alpha}/I_{A_{\alpha}}^n\}_{n \in \N})\to \hat{F}( \{P_kB_{\alpha}/I_{B_{\alpha}}^n\}_{n \in \N})
\] 
is a weak equivalence, since $F$ is homotopy-preserving. Thus $F''$ (and hence $F'$) is homotopy-preserving.

If $A \in s\cN_R^{\flat}$, note that
\[
F'(A) = \LLim_{\alpha} F(A_{\alpha}) \simeq F(A),
\]
by nilpotence and almost finite presentation, respectively, noting that as in the proof of Theorem \ref{lurierep}, conditions (\ref{afp1a}), (\ref{afp1b}) and (\ref{afp2}) ensure almost finite presentation of $F$. Thus $F\simeq  F'|_{s\cN_R^{\flat}}$; in particular, this ensures that $\DD^i_x((F'), M)\cong \DD^i_x(F, M) $.

Since $P_kA= \LLim  P_kA_{\alpha}$ (for $A_{\alpha}$ as above), it follows immediately that $F$ is nilcomplete. Likewise, $\pi^0F$ automatically preserves filtered  colimits, as do  the functors $\DD^i_x(F, -):\Mod_A \to \Ab$. Therefore $F'$ satisfies the conditions of Corollary \ref{lurierep2}.

Finally, it remains to show that $F'$ is uniquely determined by $F$. Assume that we have some geometric derived stack $G:  s\Alg_R \to \bS $, almost of finite presentation, with $G|_{s\cN_R^{\flat}}\simeq F$. Then, since $G$ is nilcomplete and almost of finite presentation, we must have
\begin{eqnarray*}
G(A)  &\simeq& {\Lim_k}^h G(P_kA)\\
&\simeq& {\Lim_k}^h  \LLim_{\alpha} G( P_kA_{\alpha})\\
& \simeq& {\Lim_k}^h  \LLim_{\alpha} G( P_k\hat{A}_{\alpha}),
\end{eqnarray*}
where we write $A= \LLim_{\alpha} A_{\alpha}$ as a filtered colimit of finitely generated subalgebras, and the final isomorphism comes from the weak equivalence $A_{\alpha}\to \hat{A}_{\alpha}$ of   \cite{stacks2} Theorem \ref{stacks-fthm}.

Now,  if we take  an inverse system $\{B_i\}_i$ in $s\Alg$ in which the morphisms $B_i \to B_j$ induce isomorphisms $ \pi_0B_i \to\pi_0B_j$, then  $G(\Lim^h B_i) \simeq \Lim^hG(B_i)$ (as $G$ is a   geometric derived stack, so has an atlas as in \cite{stacks2} Theorem \ref{stacks-relstrict}). In particular,
\begin{eqnarray*}
G( P_k\hat{A}_{\alpha})&=& G(\Lim_n  P_k\hat{A}_{\alpha}/(I_{A_{\alpha}}^n))\\
&\simeq& {\Lim_n}^hG(P_k\hat{A}_{\alpha}/(I_{A_{\alpha}}^n))\\
&=& {\Lim_n}^hF(P_k\hat{A}_{\alpha}/(I_{A_{\alpha}}^n))\\
&=& \hat{F}(P_k\hat{A}_{\alpha}).
\end{eqnarray*}
Thus
\[
G(A) \simeq {\Lim_k}^h  \LLim_{\alpha}\hat{F}(P_k\hat{A}_{\alpha}),
\]
as required.
 \end{proof}

\begin{remark}
Note that if we replace $d\cN_R^{\flat}$ with $s\cN_R$ or  $dg_+\cN_R$, then the theorem remains true, provided we impose the additional condition that $F$ be nilcomplete, in the sense that for all $A$, the map $F(A) \to \Lim^h_k F(P_kA)$ is a weak equivalence. 
\end{remark}

\subsection{Covers}

We end this section with a criterion which allows us to verify the key representability properties on formally \'etale covers.

\begin{definition}
 A transformation    $\alpha:F \to G$  of functors $F, G :  d\cN^{\flat} \to \bS$ is said to be  homotopy formally \'etale if for all square-zero extensions $A \to B$, the map
\[
F(A) \to F(B)\by^h_{G(B)}G(A)        
\]
  is an equivalence.  
\end{definition}

\begin{proposition}\label{hfettransfer}
Let $\alpha:F \to G$ be a homotopy formally \'etale morphism of functors $F, G :  d\cN^{\flat} \to \bS$. If $G$ is homotopy-homogeneous (resp. homotopy-preserving), then so is $F$. Conversely, if  $\alpha$ is surjective (in the sense that $\pi_0F(A) \onto \pi_0G(A)$ for all $A$) and $F$ is  homotopy-homogeneous (resp. homotopy-preserving), then so is $G$.
\end{proposition}
\begin{proof}
Take a square-zero extension $A \to B$, and a morphism $C \to B$, noting that $A\by_BC \to C$ is then another square-zero extension. 
Since $\alpha$ is homotopy formally \'etale,
\begin{eqnarray*}
F(A\by_BC)&\simeq& G(A\by_BC)\by^h_{GC}FC\\ 
FA\by^h_{FB}FC&\simeq& [G(A)\by^h_{G(B)}F(B)]\by^h_{FB}FC\\
&\simeq& G(A)\by^h_{G(B)}FC\\
&\simeq& (GA\by_{GB}^hGC)\by^h_{GC}FC.
\end{eqnarray*}
Thus homogeneity of $G$ implies homogeneity of $F$, and if $\pi_0FC \to \pi_0GC$ is surjective for all $C$, then homogeneity of $F$ implies homogeneity of $G$.

Now take a tiny acyclic extension $A \to B$ in $ d\cN^{\flat}$. Since $\alpha$ is homotopy formally \'etale,
\[
F(A)\simeq G(A)\by^h_{G(B)}F(B),
\]
so if $G$ is homotopy-preserving, then $F$ maps tiny acyclic extensions to weak equivalences. By Lemma \ref{wtiny} and the proof of Lemma \ref{properw}, this implies that $F$ is homotopy-preserving. If $\pi_0F(B) \to \pi_0G(B)$ is surjective for all $B$, then the converse holds.
\end{proof}

\section{Pre-representability}\label{prerepsn}

\subsection{Simplicial structures}

\begin{definition}\label{salgstr}
Define  simplicial structures (in the sense of \cite{sht} Definition II.2.1) on $s\Alg_R$  and $\ind(\widehat{FGs\Alg_R})$ as follows. For  $A\in s\Alg_R$ and $K \in \bS$,  $A^K$ is defined by 
$$
(A^K)_n:= \Hom_{\bS}(K \by \Delta^n, A).
$$ 
Then for $A  \in \ind(\widehat{FGs\Alg_R})$, $A^K$ is uniquely determined via Lemma \ref{hatful} by 
the property 
that
$U(A^K)= (UA)^K$.

Spaces  $\HHom(A,B) \in \bS$ of morphisms are then given by
$$
\HHom(A, B)_n:= \Hom(A, B^{\Delta^n}).
$$
\end{definition}

We need to check that this is well-defined:
\begin{lemma}
For $A \in \ind(\widehat{FGs\Alg_R})$  and $K \in \bS$, we have $A^K \in \ind(\widehat{FGs\Alg_R}) $. Moreover, if  
$A \to \pi_0A$ is a nilpotent extension, then so is $A^K \to \pi_0(A^K)$.
\end{lemma}
\begin{proof}
$A^K$ can be expressed as the limit
$$
\lim_{\substack{\lla \\ (\Delta^n \xra{f} K) \in \Delta\da K }} A^{\Delta^n};
$$
since the inclusion functor $U : \ind(\widehat{FGs\Alg_R})\to s\Alg_R$ is a right adjoint, it preserves arbitrary limits,
so  it suffices to show that $A^{\Delta^n} \in \ind(\widehat{FGs\Alg_R})$. 

Write $A:= \LLim_{\alpha} A_{\alpha}$, for $A_{\alpha} \in \widehat{FGs\Alg_R}$. Since $\Delta^n$ is finite, we have $A^{\Delta^n}= \LLim_{\alpha} A_{\alpha}^{\Delta^n}$, so we may assume that $A \in  \widehat{FGs\Alg_R}$.

The exact sequence $0 \to I_A \to A \to \pi_0A \to 0$ gives an exact sequence $0 \to I_A^{\Delta^n} \to A^{\Delta^n} \to \pi_0A \to 0$ (as $(\pi_0A)^{\Delta^n}= \pi_0A$, since $\Delta^n$ is connected). Since $\Delta^n$ is contractible, $\pi_0(I_A^{\Delta^n})= \pi_0(I_A)=0$, so  $I_{A^{\Delta^n}}= I_A^{\Delta^n}$. Hence
\[
\Lim_m (A^{\Delta^n}/ I_{A^{\Delta^n}}^m)= \Lim_m (A^{\Delta^n}/ (I_{A}^{\Delta^n})^m)= \Lim (A/I_A^m)^{\Delta^n}= A^{\Delta^n},
\]
so $A^{\Delta^n} \in \widehat{FGs\Alg_R}$.

Finally, if $I_A^m=0 $, then $(I_{A}^{\Delta^n})^m=0$, so $I_{A^{\Delta^n}}^m=0$ for all $n$, and hence $I_{A^K}^m=0$ for all $K \in \bS$. 
 \end{proof}

In fact, this makes $\ind(\widehat{FGs\Alg_R})$ into a simplicial model category in the sense of \cite{sht} Ch. II (with $U : \ind(\widehat{FGs\Alg_R})\to s\Alg_R$ becoming a simplicial right Quillen equivalence).
Although the same is not true  for $dg_+\Alg_R$ or  $\ind(\widehat{FGdg_+\Alg_R})$, we now show that they carry  compatible weak simplicial structures. 
\begin{definition}
Explicitly, we say that a model category $\C$ has a weak simplicial structure if we  have the following data:
\begin{enumerate}
\item a functor $\HHom_{\C}: \C^{\op} \by \C \to \bS$ such that $\HHom_{\C}(A,B)_0 = \Hom_{\C}(A,B)$. 
 \item 
a functor  $(f\bS)^{\op}\by \C \to \C$ (where $f\bS$ is the category of finite simplicial sets), denoted by $(K,B) \mapsto B^K$, with natural isomorphisms
$$
\Hom_{\C}(A, B^K)\cong \Hom_{\bS}(K,\HHom_{\C}(A,B)).
$$
\end{enumerate}
These must satisfy the property (known as SM7) that if $i:A \to B$ is a cofibration in $\C$, and $p:X \to Y$ a fibration, then
\[
\HHom_{\C}(B, X) \to \HHom_{\C}(A, X)\by_{\HHom_{\C}(A, Y)}\HHom_{\C}(B, Y)
\] 
is a fibration in $\bS$ which is trivial whenever either $i$ or $p$ is a weak equivalence.
\end{definition}
This means that $\C$ satisfies all of the axioms of a simplicial model category from \cite{sht} Ch. II except for conditions (2) and (3) of  Definition II.2.1 (which require that for all objects $A \in \C$,  the functors $\HHom_{\C}(A, -):\C \to \bS$ and $\HHom_{\C}(-, A):\C^{\op} \to \bS  $ have left adjoints). 

Note that  this is enough to ensure that $\C$ is still  a simplicial model category in the sense of \cite{QHA}. 

\begin{lemma}\label{simplicialstr}
 The model categories $dg_+\Alg_R$ and $\ind(\widehat{FGdg_+\Alg_R})$ carry weak simplicial structures.
\end{lemma}
\begin{proof}
First set $\Omega_n=\Omega(\Delta^n)$ to be the  cochain algebra 
$$
\Q[t_0, t_1, \ldots, t_n,dt_0, dt_1, \ldots, dt_n ]/(\sum t_i -1, \sum dt_i)
$$  
of rational differential forms on the $n$-simplex $\Delta^n$.
 These fit together to form a simplicial complex $\Omega_{\bt}$ of DG-algebras, and we define $A^{\Delta^n}$ as the good truncation $A^{\Delta^n}:= \tau_{\ge 0}(A \ten \Omega_n)$.  Note that this construction only commutes with finite limits, so only extends to define $A^K$ for finite simplicial sets $K$, and does not have a left adjoint.

For $A \in \widehat{FGdg_+\Alg_R}$, we replace $A^K$ with its completion over $\H_0(A^K)$, and extend this construction to 
$\ind(\widehat{FGdg_+\Alg_R})$ in the obvious way.

That these have the required properties follows because the matching maps $\Omega_n \to M_n\Omega=\Omega(\pd\Delta^n)$ are surjective. Explicitly,
$$
M_n\Omega \cong 
\Omega_n/( t_0\cdots t_n, \sum_i t_0\cdots t_{i-1}  (dt_i) t_{i+1} \cdots t_n).
$$ 
\end{proof}

\begin{definition}
Although the categories $s\cN_R^{\flat}$ and $dg_+\cN_R^{\flat}$ are not model categories, we endow them with weak simplicial structures inherited from $s\Alg_R$ and $dg_+\Alg_R$, respectively. The key observation is that for $K \in f\bS$ and $A \in d\cN^{\flat}$, the object $A^K$ lies in $ d\cN^{\flat}$.
\end{definition}

\subsection{Deriving functors}

\begin{definition}
Given
a functor $F: d\cN^{\flat} \to \bS$, we  define  a functor $\underline{F}: d\cN^{\flat}\to s\bS$ to  the category of bisimplicial sets by 
$$
\underline{F}(A)_{n} :=  F(A^{\Delta^n}).
$$
\end{definition}

For a functor $F:\C \to \Set$, we will abuse notation by also writing $F: \C\to \bS$ for the composition $ \C\xra{F} \Set \to \bS$. 

\begin{proposition}\label{smatchnew}
If $F:d\cN^{\flat} \to \bS$ is homotopy-homogeneous, then
for $A \to B$ an acyclic little extension  in $d\cN^{\flat}$    and $K \in \bS$ finite, the map
\[
F(A^K)  \to( M_K^h\uline{F}(A))\by_{(M_K^h\uline{F}(B))}^hF(B^K) 
\]
is a weak equivalence in $\bS$, where $M_K^h$ denotes the Reedy homotopy $K$-matching object. 
\end{proposition}
\begin{proof}
We prove this by induction on the dimension of $K$. If $K$ is dimension $0$ (i.e. discrete), then the map is automatically an equivalence, as 
\[
M_K^h\uline{F}(A)= \uline{F}(A)_{0}^K = F(A^K).
\] 

Now assume the statement holds for all finite simplicial sets of dimension $<n$, take $K$ of dimension $n$, and let $K':= \sk_{n-1}K$, the $(n-1)$-skeleton. Thus there is a pushout square
\[
\begin{CD}
(\pd\Delta^n \by N_nK) \sqcup (\Delta^n \by L_nK) @>>>      \Delta^n \by K_n\\
@VVV @VVV\\
K' @>>> K,
\end{CD}
\]
where $L_nK$ is the $n$th latching object and  $N_nK= K_n-L_nK$. Hence we have a pullback square
\[
\begin{CD}
A^K  @>>> B^K\by_{B^{K'}} A^{K'}\\
@VVV @VVV\\
B^K\by_{B^{(\Delta^n \by K_n)}}A^{(\Delta^n \by K_n)} @>>> B^K\by_{[B^{(\pd \Delta^n \by N_nK)}\by B^{(\Delta^n \by L_nK)}]}[A^{(\pd\Delta^n \by N_nK)}\by A^{(\Delta^n \by L_nK)}].
\end{CD}
\]

Now, since $A \to B$ is an acyclic little extension, the map $A^{\Delta^n}  \to A^{\pd\Delta^n} \by_{B^{\pd\Delta^n}}B^{\Delta^n}$ is a square-zero extension, so the bottom map in the diagram above is  a square-zero extension, giving a homotopy pullback square
\[
\begin{CD}
F(A^K)  @>>> F(B^K)\by_{F(B^{K'})}^h F(A^{K'})\\
@VVV @VVV\\
F(B^K)\by_{F(B^{\Delta^n})^{K_n}}^hF(A^{\Delta^n})^{K_n} @>>> F(B^K)\by^h_{[F(B^{\pd\Delta^n})^{N_nK}\by F(B^{\Delta^n})^{L_nK}]}[F(A^{\pd\Delta^n})^{N_nK}\by F(A^{\Delta^n})^{L_nK}].
\end{CD}
\]
Here, the top right isomorphism comes from $A^{K'} \onto B^{K'}$, the bottom left from $A^{\Delta^n} \onto B^{\Delta^n}$, and the bottom right from  $A^{\Delta^n} \onto B^{\Delta^n}$ and from $A^{\pd\Delta^n} \onto B^{\pd\Delta^n}$; these are all square-zero extensions and $F$ is homotopy-homogeneous.

By induction (using $F(A^{K'})  \simeq (M_{K'}^h\uline{F}(A))\by_{(M_{K'}^h\uline{F}(B))}^hF(B^{K'})$ and $F(A^{\pd \Delta^n})  \simeq (M_n^h\uline{F}(A))\by_{(M_n^h\uline{F}(B))}F(B^{\pd \Delta^n})$), we can rewrite this as saying that the following square is a homotopy pullback
\[
\begin{CD}
F(A^K)  @>>> F(B^K)\by_{M_{K'}^h\uline{F}(B)}^h M_{K'}^h\uline{F}(A)\\
@VVV @VVV\\
F(B^K)\by_{F(B^{\Delta^n})^{K_n}}^hF(A^{\Delta^n})^{K_n} @>>> F(B^K)\by_{[M_n^h\uline{F}(B)^{N_nK}\by F(B^{\Delta^n})^{L_nK}]}[M_n^h\uline{F}(A)^{N_nK}\by F(A^{\Delta^n})^{L_nK}].
\end{CD}
\]
Now just observe that  this pullback   defines $F(B^K)\by_{M_K^h\uline{F}(B)}M_K^h\uline{F}(A)$, as required.
\end{proof}

\begin{definition}
Say that a   functor $F: d\cN^{\flat} \to \bS$ is homotopy-surjecting if for all  tiny acyclic extensions  $A \to B$, the map
$$
\pi_0F(A) \to \pi_0F(B)
$$
is surjective. 
\end{definition}

\begin{definition}
Define  $\bar{W}:s\bS \to \bS$ to be the right adjoint to  Illusie's total $\Dec$ functor given by $\DEC(X)_{mn}= X_{m+n+1}$. Explicitly,
\[
 \bar{W}_p(X) = \{(x_0, x_1, \ldots, x_p) \in
 \prod^p_{i=0} X_{i,p-i} | \pd^v_0 x_i = \pd^h_{i+1}x_{i+1},\, \forall 0 \le i <p\}
\]
with operations
\begin{eqnarray*}
 \pd_i(x_0, \ldots, x_p) &=& (\pd^v_i x_0, \pd^v_{i-1}x_1, \ldots , \pd^v_1 x_{i-1}, \pd^h_ix_{i+1}, \pd^h_i x_{i+2}, \ldots, \pd^h_i x_p),\\
\sigma_i(x_0, \ldots, x_p) &=& (\sigma^v_i x_0,\sigma^v_{i-1}x_1, \ldots , \sigma^v_0 x_i, \sigma^h_i x_i, \sigma^h_i x_{i+1}, \ldots, \sigma^h_i x_p).
\end{eqnarray*}
\end{definition}

In \cite{CRdiag},  it is established that the canonical natural transformation
\[
\diag X \to \bar{W}X
\]
from the diagonal is a weak equivalence for all $X$.

\begin{corollary}\label{settotop} 
If a homotopy-homogeneous functor $F:d\cN \to \bS$ is  homotopy-surjecting, then   the   functor  $\bar{W}\underline{F}:d\cN \to \bS$ is  homotopy-preserving. 
\end{corollary}
\begin{proof}
Consider the homotopy matching maps (for the Reedy model structure on bisimplicial sets) 
$$
\underline{F}(A)_n \to \underline{F}(B)_n\by_{M_{\pd \Delta^n}^h \underline{F}(B)}M_{\pd \Delta^n}^h\underline{F}(A)
$$
of 
\[
 \uline{F}(A) \to \uline{F}(B),
\]
 for an acyclic little extension $A \to B$. By Lemma \ref{wtiny}, we may replace tiny acyclic extensions with little acyclic extensions in the definition of homotopy-surjections.

By Proposition \ref{smatchnew}, the map above is weakly equivalent to 
$$
F(A') \to F(B'),
$$
where $A'= A^{\Delta^n}, B'= B^{\Delta^n}\by_{ B^{\pd\Delta^n}}A^{\pd\Delta^n}$. Now, $A' \to B'$ is a little acyclic extension, so the homotopy  matching maps of $\uline{F}(A) \to \underline{F}(B)$  are surjective on $\pi_0$  (as $\alpha$ is  homotopy-surjecting). 

For  any Reedy fibrant replacement $f:R \to \underline{F}(B)$ of $\uline{F}(A) \to \underline{F}(B)$, the homotopy  matching maps must also be surjective on $\pi_0$. However, for Reedy fibrations, matching objects model homotopy matching objects, so $f$  
is 
a Reedy surjective fibration, and hence a horizontal levelwise trivial fibration (the matching maps being surjective). It is therefore a diagonal weak equivalence by \cite{sht}  Proposition IV.1.7, and \cite{CRdiag} then shows that  $\bar{W}f$ is also a weak equivalence. Lemma \ref{properw} then implies that $\bar{W}F$ preserves all weak equivalences.
\end{proof}

\begin{proposition}\label{spmatchnew}
If $F:d\cN^{\flat} \to \bS$ is homotopy-homogeneous, then
for $A \to B$ a little extension  in $d\cN^{\flat}$    and $K$   a contractible  finite simplicial set, the map
\[
F(A^K)  \to( M_K^h\uline{F}(A))\by_{(M_K^h\uline{F}(B))}^hF(B^K) 
\]
is a weak equivalence in $\bS$. 
\end{proposition}
\begin{proof}
We adapt the proof of  Proposition \ref{smatchnew},
proceeding by induction on the dimension of $K$. If $K$ is of dimension $0$, the statement is automatically true.

 For  any contractible finite simplicial set $K$, any  morphism $\Delta^0 \to K$ can be expressed as an iterated pushout of anodyne extensions $\L^{m,k} \to \Delta^m$. In particular, if $K$ has dimension $n$, there is a contractible  simplicial set $K' \subset K$ of dimension $n-1$, with the map $K'\to K$ an iterated pushout of the maps $\L^{n,k} \to \Delta^n$ for various $k$.  The proposition holds  by induction for $K'$ and $\L^{n,k}$, and is automatically  satisfied by $\Delta^n$. 

Since the map $A^{\Delta^n} \to B^{\Delta^n}\by_{B^{\L^n,k}}A^{\L^n,k}$ is an acyclic little extension,
 the proof of  Proposition \ref{smatchnew} adapts to show that the proposition is satisfied by $K$, as required.
\end{proof}

\begin{corollary}\label{settotopb}
If a homotopy-homogeneous functor $F:d\cN \to \bS$ is  homotopy-surjecting, then the   functor  $\bar{W}\underline{F}:d\cN \to \bS$ is homotopy-homogeneous.
\end{corollary}
\begin{proof}
Take a square-zero little extension $A \to B$; by Proposition \ref{spmatchnew}, the relative homotopy partial matching  object 
\[
 M_{\L^{n,k}}^h\uline{F}(A)\by^h_{M^h_{\L^{n,k}}\uline{F}(B)}\uline{F}(B)_n
\]
 is 
\[
F(A^{\L^{n,k}})\by^h_{F(B^{\L^{n,k}})}F(B^{\Delta^n}).
\]
 
Since $A^{\Delta^{n}} \to A^{\L^{n,k}}\by_{B^{\L^{n,k}}}B^{\Delta^n}$  is an  acyclic little extension, homotopy-surjectivity of $F$ thus implies that the homotopy partial matching map
\[
\uline{F}(A)_n \to M_{\L^{n,k}}^h\uline{F}(A)\by^h_{M^h_{\L^{n,k}}\uline{F}(B)}\uline{F}(B)_n
\]
 gives a surjection on $\pi_0$. 

If we take a Reedy fibrant replacement $R$ for  $\uline{F}(A)$ over $\uline{F}(B)$, this says that
\[
R_n \to M_{\L^{n,k}}R\by_{M_{\L^{n,k}}\uline{F}(B)}\uline{F}(B)_n
\] 
is surjective on $\pi_0$ --- since it is (automatically) a fibration, this implies that it is surjective levelwise.

Thus $f:R \to \uline{F}(B)$ is a Reedy fibration  and a horizontal levelwise Kan fibration, so \cite{sht} Lemma IV.4.8 implies that $\diag f$ is a fibration, so for any map $C \to B$,
\begin{eqnarray*}
(\diag\uline{F}(A))\by^h_{(\diag\uline{F}(B))}(\diag\uline{F}(C)) &\simeq& (\diag R)\by_{(\diag\uline{F}(B))}(\diag\uline{F}(C))\\
&=& \diag( R\by_{(\uline{F}(B)}\uline{F}(C)  )\\
&\simeq& \diag( \uline{F}(A)\by_{\uline{F}(B)}^h\uline{F}(C))\\
&\simeq&\diag \uline{F}(A \by_BC),
\end{eqnarray*}
the penultimate equivalence following because $R \to \uline{F}(B)$ is a Reedy fibrant replacement for $\uline{F}(A)$, and the final one because $F$ is homotopy-homogeneous and $A \to B$ is square-zero. 

Finally, \cite{CRdiag} shows that $\bar{W}X$ and $\diag X$ are weakly equivalent for all $X$, so
\[
(\bar{W}\uline{F}(A))\by^h_{(\bar{W}\uline{F}(B))}(\bar{W}\uline{F}(C)) \simeq \bar{W}\uline{F}(A \by_BC).
\]
Any  square-zero extension $A \to B$ in $d\cN$ with kernel $K$ can be expressed as the composition of the little extensions $A/(I_A^{n+1}K) \to A/(I_A^{n}K)$, making $\bar{W}\uline{F}$ homotopy-homogeneous.
\end{proof}

\begin{lemma}\label{ulinec}
For a homotopy-preserving functor $F:d\cN \to \bS$, the natural transformation $F \to \bar{W}\uline{F}$  is a weak equivalence.
\end{lemma}
\begin{proof}
The transformation comes from applying $\bar{W}$ to the maps $F(A) \to \uline{F}(A)$ of bisimplicial sets coming from  the canonical maps $A \to A^{\Delta^n}$.

Since $A \to A^{\Delta^n}$ is a weak equivalence, the maps $F(A) \to \uline{F}(A)$ are also weak equivalences levelwise, so $F=\bar{W}F \to \bar{W}\uline{F}$ is a weak equivalence (as $\bar{W}$ sends levelwise weak equivalences to weak equivalences).
\end{proof}

\subsection{Representability}

\begin{definition}\label{pretotcohodef}
Given a homotopy-surjecting homotopy-homogeneous functor $F: d\cN_R^{\flat} \to \bS$, $A \in  \Alg_{\H_0R}$, $x \in F_0(A)$, and an     $A$-module $M$, define $\DD^{i}_x(F,M)$ as follows.

For $i\le 0$, set 
\[
\DD^i_x(F,M):= \pi_{-i}(F(A\oplus M)\by^h_{F(A)}\{x\}). 
\]
For $i >0$, set 
\[
\DD^i_x(F,M):= \pi_0(F(A\oplus M[-i])\by^h_{F(A)}\{x\})/\pi_0(F(A\oplus \cone(M)[1-i])\by^h_{F(A)}\{x\} ). 
\]

Note that homotopy-homogeneity of $F$ ensures that these are abelian groups for all $i$, and that the multiplicative action of $A$ on $M$ gives them the structure of $A$-modules.
\end{definition}

\begin{lemma}\label{totcohoc}
For all $F, A, M$ as above, there are canonical isomorphisms
\[
\DD^i_x(F,M) \cong \DD^i_x(\bar{W}\uline{F},M),
\]
where the group on the left-hand side is defined as in  Definition \ref{pretotcohodef}, and that on the right as in Definition \ref{totcohodef}.

In particular, if $F$ is homotopy-preserving, then Definitions  \ref{pretotcohodef} and  \ref{totcohodef} are consistent.
\end{lemma}
\begin{proof}
We begin by noting that $\bar{W}\uline{F}$ is indeed homotopy-preserving and homotopy-homogeneous, by Corollaries \ref{settotop} and \ref{settotopb}. Since  $F(A)= \bar{W}\uline{F}(A)$ for all $ A \in \Alg_{\H_0R}$, it follows immediately that $ \DD^i_x(F,M) \cong \DD^i_x(\bar{W}\uline{F},M)$ for all $i \le 0$. Now for $i >0$,
\begin{eqnarray*}
\DD^i_x(\bar{W}\uline{F},M)&=& \pi_0 (\bar{W}(T_x(\uline{F}/R)(M[-i])) \\   
&=& \pi_0(T_x(F/R)(M[-i]))/ \pi_0 (F((A\oplus M[-i])^{\Delta^1}\by_{A^{\Delta^1}}A)\by^h_{F(A)}\{x\})\\
&=&   \pi_0(T_x(F/R)(M[-i]))/ \pi_0(T_x(F/R)(M[-i]^{\Delta^1})),  
\end{eqnarray*}
where the quotient is taken by the map $\pd_0-\pd_1$ coming from the projections $M^{\Delta^1} \to M$. 
If $d\cN^{\flat}_R= s\cN^{\flat}_R$, then $M[-i]^{\Delta^1}= M[-i] \oplus \cone(M)[1-i]$, so   $\DD^i_x(\bar{W}\uline{F},M)\cong \DD^i_x(F,M)$. When $d\cN^{\flat}_R= dg_+\cN^{\flat}_R$, we have more work to do. In this case, $M[-i]^{\Delta^1} = \tau_{\ge 0} (M[-i]\ten \Omega_n)$. The key observation to make is that $M[-i] \oplus \cone(M)[1-i]$ can be expressed as a retract of $\tau_{\ge 0} (M[-i]\ten \Omega_n)$ over $M \oplus M$ given by $m \ten 1 \mapsto (m,0)$, $m\ten x_0^n \mapsto (0,m)$ for $n>0$, and   $m \ten x_0^ndx_0  \mapsto (0,dm/(n+1))$. Thus  
\[
(\pd_0-\pd_1):  \pi_0(T_x(F/R)(M[-i]^{\Delta^1}) \to \pi_0(T_x(F/R)(M[-i]))
\]
has the same image as $\pi_0(T_x(F/R)(\cone(M)[1-i]) \to \pi_0(T_x(F/R)(M[-i]))$, so $\DD^i_x(\bar{W}\uline{F},M)\cong \DD^i_x(F,M)$.

Finally, if $F$ is homotopy-preserving, then Lemma \ref{ulinec} shows that the map $F \to \bar{W}\uline{F}$ is a weak equivalence, making  the definitions  consistent.
\end{proof}

\begin{theorem}\label{lurieprerep}\label{lureisemirep}
Let $R$ be a derived G-ring admitting a dualising module (in the sense of \cite{lurie} Definition 3.6.1) and   take a functor $F: d\cN_R^{\flat} \to \bS$ satisfying the following conditions.

\begin{enumerate}
 
\item $F$ is homotopy-surjecting, i.e. it maps tiny acyclic extensions to surjections (on $\pi_0$).

\item For all discrete rings $A$, $F(A)$ is $n$-truncated, i.e. $\pi_iF(A)=0$ for all $i>n$ .

\item
$F$ is homotopy-homogeneous, i.e. for all square-zero extensions $A \onto C$ and all maps $B \to C$, the map
$$
F(A\by_CB) \to F(A)\by_{F(C)}^hF(B)
$$
is an equivalence.


\item $\pi^0F:\Alg_{\H_0R} \to \bS$ is a hypersheaf for the \'etale topology. 

\item
$\pi_0\pi^0F:  \Alg_{\H_0R} \to \Set$  preserves filtered colimits.

\item
For all $A \in \Alg_{\H_0R}$ and all $x \in F(A)$, the functors $\pi_i(\pi^0F,x): \Alg_A \to \Set$  preserve filtered colimits for all $i>0$.

\item 
for all finitely generated integral domains $A \in \Alg_{\H_0R}$, all $x \in F(A)_0$ and all \'etale morphisms $f:A \to A'$, the maps 
\[
\DD_x^*(F, A)\ten_AA' \to \DD_{fx}^*(F, A')
\]
are isomorphisms.

\item
 for all finitely generated $A \in \Alg_{\H_0R}$  and all $x \in F(A)_0$, the functors $\DD^i_x(F, -): \Mod_A \to \Ab$ preserve filtered colimits for all $i>0$.

\item for all finitely generated integral domains $A \in \Alg_{\H_0R}$  and all $x \in F(A)_0$, the groups $\DD^i_x(F, A)$ are all  finitely generated $A$-modules.

\item for all complete discrete local Noetherian  $\H_0R$-algebras $A$, with maximal ideal $\m$, the map
$$
\pi^0F(A) \to {\Lim}^h F(A/\m^r)
$$
is a weak equivalence.
\end{enumerate}
 Then $\bar{W}\uline{F}$ is the restriction to $d\cN^{\flat}_R$ of a geometric derived $n$-stack $F':s\Alg_R \to \bS$ (resp. $F':dg_+\Alg_R \to \bS$), which is almost of finite presentation. Moreover, $F'$ is uniquely determined by $F$ (up to weak equivalence).
\end{theorem}
\begin{proof}

By Corollaries \ref{settotop} and \ref{settotopb}, $\bar{W}\uline{F}$ is homotopy-preserving and homotopy-homogeneous. Since $\pi^0F= \pi^0\uline{F}$, the map   $\pi^0F \to \pi^0 \bar{W}\uline{F}$ is a weak equivalence. Lemma \ref{totcohoc} then shows that $\DD^i_x(F,M) \cong \DD^i_x(\bar{W}\uline{F},M)$, so $\bar{W}F$ satisfies all the conditions of Theorem \ref{lurierep3}.
\end{proof}

\begin{example}\label{dgexample}
If $X$ is a dg manifold (in the sense of \cite{Quot}), then the functor $X: dg_+\cN^{\flat}_R \to \Set$ given by $X(A)= \Hom(\Spec A, X)$ satisfies the conditions of Theorem \ref{lurieprerep}, so $\uline{X}:  dg_+\cN^{\flat}_R \to \bS$ is a geometric derived $0$-stack.

In fact, $\uline{X}$ is just the hypersheafification of $X$. This follows because $\uline{X}$ is a geometric derived $0$-stack, so $\uline{X}^{\sharp}= \uline{X}$, and there is thus a map $f:X^{\sharp} \to \uline{X}$. Since $X^{\sharp}$ is a geometric derived $0$-stack (as can be shown for instance by observing that it is equivalent to the derived stack  $\gpd(X)^{\sharp}$ of \cite{stacks2} \S \ref{stacks-dgstacks}), Proposition \ref{detectweak} implies that $f$ must be an equivalence.
\end{example}

This example will be adapted further in \cite{dmc}, constructing geometric derived $n$-stacks from DG Lie algebras similar to those used in \cite{Hilb} and \cite{Quot}.

\bibliographystyle{alphanum}
\bibliography{references}
\end{document}